\documentclass[oneside]{amsart}
\usepackage[utf8]{inputenc}
\usepackage{amsmath,amsthm,amssymb}
\usepackage{mathrsfs} 
\usepackage{graphicx, xcolor}
\usepackage{caption} 
\usepackage{subfiles}
\usepackage{newverbs} 
\usepackage{array} 
\usepackage{tikz}
\usepackage{braket}
\usepackage{enumerate}

\usepackage{pgfplots}
\usepgfplotslibrary{fillbetween}
\pgfplotsset{width=7cm,compat=1.9}

\usetikzlibrary{decorations.markings,arrows.meta,patterns}
\tikzset
  {midarrow/.style={decoration={markings,mark=at position 0.5 with
     {\arrow[xshift=2pt]{Latex[length=4pt,#1]}}},postaction={decorate}}
  }

\usepackage[colorlinks,citecolor=blue,linkcolor=red!70!black]{hyperref}
\usepackage[nameinlink]{cleveref}

\usepackage{comment}

\usepackage[marginpar=5cm,left=5cm,right=5cm]{geometry}

\usepackage{amssymb} 
\usepackage{graphicx} 

\newcommand{\asterisk}[2]{\begin{tikzpicture}
    \node[circle,fill=red,scale=0.3] at (360:0mm) (center) {};
    \foreach \n in {1,...,#1}{
        \node[circle,fill=blue,scale=0.3] at ({\n*360/#1}:#2cm) (n\n) {};
        \draw (center)--(n\n);
    }
    \node at (0,-#2*1.5) {$\ast_{#1}$}; 
\end{tikzpicture}}

\newverbcommand{\CMRverb}{\tiny\color{blue}}{}
\newverbcommand{\GRverb}{\tiny\color{teal}}{}
\newverbcommand{\STverb}{\tiny\color{cyan}}{}
\newverbcommand{\Pverb}{\tiny\color{violet}}{}
\newverbcommand{\Averb}{\tiny\color{brown}}{}

\theoremstyle{plain}
\newtheorem{THM}{Theorem}[section]
\newtheorem{THME}{Theorem}[section]
\newtheorem{PROP}[THM]{Proposition}
\newtheorem{LEM}[THM]{Lemma}

\newtheorem{FACT}[THM]{Fact}
\newtheorem{CLAIM}{Claim}

\theoremstyle{definition}
\newtheorem{DEF}[THM]{Definition}
\newtheorem{RMK}[THM]{Remark}
\newtheorem{EX}[THM]{Example}

\DeclareMathOperator{\id}{id}

\DeclareMathOperator{\Var}{Var}
\DeclareMathOperator{\Out}{Out}
\DeclareMathOperator{\Aut}{Aut}
\DeclareMathOperator{\Inn}{Inn}

\newcommand{\eg}{\textit{e.g. }}

\newcommand{\dfn}{\mathrel{\mathop:}=}

\renewcommand{\>}{\rangle}
\newcommand{\C}{\mathbb{C}}

\newcommand{\F}{\mathbb{F}}
\newcommand{\G}{\Gamma}
\newcommand{\cO}{\mathcal{O}}
\renewcommand{\P}{\mathbb{P}}
\newcommand{\Q}{\mathbb{Q}}
\newcommand{\R}{\mathbb{R}}

\newcommand{\Z}{\mathbb{Z}}

\renewcommand{\bar}{\overline}
\renewcommand{\|}{|\!|} 

\title{Thurston's Theorem: Entropy in Dimension One}
\author[Dickmann, Domat, Hill, Kwak, Ospina, Patel, Rechkin]{Ryan Dickmann, George Domat, Thomas Hill,\\ Sanghoon Kwak, Carlos Ospina, Priyam Patel, Rebecca Rechkin}

\begin{document}
\maketitle

\begin{abstract}
    In his paper \cite{Thurston2014Entropy}, Thurston shows that a positive real number $h$ is the topological entropy for an ergodic traintrack representative of an outer automorphism of a free group if and only if its expansion constant $\lambda = e^h$ is a weak Perron number. This is a powerful result, answering a question analogous to one regarding surfaces and stretch factors of pseudo-Anosov homeomorphisms. However, much of the machinery used to prove this seminal theorem on traintrack maps is contained in the part of Thurston's paper on the entropy of postcritically finite interval maps and the proof is difficult to parse. In this expository paper, we modernize Thurston's approach, fill in gaps in the original paper, and distill Thurston's methods to give a cohesive proof of the traintrack theorem. Of particular note is the addition of a proof of ergodicity of the traintrack representatives, which was missing in Thurston's paper. 
\end{abstract}

\section{Introduction}

Topological entropy describes the complexity of the orbit structure of a dynamical system, and is an invariant of topological conjugacy classes. A classical problem in dynamics is to characterize the numbers that can arise as the topological entropy for a particular family of dynamical systems. In his paper \cite{Thurston2014Entropy}, Thurston proved that a positive real number $h$ is the topological entropy of a postcritically finite self-map of the unit interval if and only if $\lambda = e^h$ is a weak Perron number, i.e., $\lambda$ is an algebraic integer that is at least as large as the absolute value of any conjugate of $\lambda$. He uses the tools developed for studying interval maps to then prove the following theorem about outer automorphisms of free groups, answering a prominent question in geometric group theory.

\begin{THM}[{\cite[Theorem 1.9]{Thurston2014Entropy}}]
    \label{thm:thurston_main}
    A positive real number $h$ is the topological entropy for an ergodic traintrack representative of an outer automorphism of a free group if and only if $\lambda = e^{h}$ is an algebraic integer that is at least as large as the absolute value of any conjugate of $\lambda$, i.e. $\lambda$ is a weak Perron number.
    \end{THM}
    
    Unfortunately, Thurston fell ill while writing a draft of the manuscript \cite{Thurston2014Entropy} and there are some gaps in the final version of the paper. In this expository article, we modernize Thurston’s approach, fill in gaps in the original paper, and distill Thurston’s methods to give a cohesive proof of the traintrack theorem that is especially readible for geometric group theorists. Though the motivation for the proof of \Cref{thm:thurston_main} comes from the dynamics of postcritically finite interval maps, we exclude details about such maps below since the purpose of this paper is to give a complete and concise proof of Thurston’s traintrack theorem.

To understand the content of \Cref{thm:thurston_main}, recall that for a finitely generated free group $\mathbb{F}_n$,
the \textbf{outer automorphism group} of $\mathbb{F}_n$ is $\Out(\mathbb{F}_n) = \Aut(\mathbb{F}_n)/ \Inn(\mathbb{F}_n)$. Bestvina and Handel \cite[Theorem 1.7]{Bestvina1992train} showed that certain outer automorphism of a free group can be represented by a special map between graphs called a \textbf{traintrack map} (see \cite[Chapter 6.3]{OHGGT}).  Roughly, a traintrack map is a continuous graph map which has particularly nice cancellation properties with respect to iterations. A traintrack map $f:\Gamma \to \Gamma$ is called \emph{irreducible} if $f$ does not admit an invariant proper subgraph which is not a tree. We formally define and give all relevant background on traintrack maps in \Cref{ssec:BG_Traintrack_Structures}. The notion of irreducibility is compatible with that of ergodicity of the traintrack map. In particular, the map being ergodic as a dynamical system essentially means the system cannot be reduced or factored into smaller components, which in this case would be proper subgraphs. Thus, ergodicity of the traintrack map implies irreduciblility.

Since their introduction, traintrack maps have become a standard tool for understanding the geometry and dynamics of automorphisms of free groups. It is also easy to calculate the expansion constant of a traintrack map since traintrack maps eliminate backtracking when iteratively applying the map to an edge or edge path. In fact, the expansion constant of the traintrack map $f$ can be calculated by finding the Perron-Frobenius eigenvalue, $\lambda$, of the transition matrix for $f$. The topological entropy of $f$ is then exaclty $\log(\lambda)$. See \Cref{ssec:BG_Expansion_Constants}, \Cref{ssec:BG_PF_Theorem}, and \cite[Remark 1.8]{Bestvina1992train} for more details.

Thus, one direction of \Cref{thm:thurston_main} follows almost immediately from \Cref{thm:PF_Theorem} (the Perron-Frobenius Theorem). In particular, an ergodic traintrack representative of an outer automorphism of a free group has a transition matrix with a positive leading eigenvalue $\lambda$ that is not smaller than the magnitude of the other eigenvalues (see the last paragraph of \Cref{ssec:BG_PF_Theorem}). Therefore, the entropy of the traintrack map is $h(f) = \log(\lambda)$ and $e^h = \lambda$ is a weak Perron number as desired. Proving the other direction is far more difficult, but we provide a sketch of Thurston's (and our) argument here. 

Fix a Perron number $\lambda$, i.e. an algebraic integer that is \emph{strictly} larger than the absolute value of its conjugates. Thurston uses two main ingredients to construct a traintrack map with growth rate $\lambda$. First, he defines a collection of prototype traintrack maps $\phi_n$ for all odd, positive integers $n$. Second, he defines a star map on a star graph that is \emph{uniformly $\lambda$-expanding}, called $f_\lambda$ below. (For the definition of star graphs and star maps, see \Cref{ssec:AST_Definition}; these are called \emph{asterisk} maps on \emph{asterisk} graphs in \cite{Thurston2014Entropy}.) This is a delicate process that requires an understanding of the arithmetic of Perron numbers. In fact, there is a somewhat serious number theoretic error in the version of \Cref{lem:even} (The Even Lemma) that appears in Thurston's paper \cite[p.359, proof of Theorem 6.2]{Thurston2014Entropy}.
We correct this error and correct the construction
of the star maps accordingly (see \Cref{rmk:n0error} and \Cref{ssec:AST_Uniform_Expanding} respectively).
It was noted by the referee that Richard Webb and his reading group at Cambridge were also aware of the error and how to fix the argument.

Next, given the star map $f_\lambda : \Gamma \to \Gamma$, Thurston defines the split graph $S(\Gamma)$ and the split map $S(f_\lambda): S(\Gamma) \to S(\Gamma)$, which are defined using \emph{both} the prototype traintrack maps $\phi_n$ and the $\lambda$-expanding star map $f_\lambda$.
The desired traintrack map for \Cref{thm:thurston_main} is $S(f_\lambda)$, and the definition of split maps ensures that the expansion constant of $S(f_\lambda)$ is $\lambda$ because the expansion constant of $f_\lambda$ is $\lambda$. It then remains to show that the traintrack map actually represents an outer automorphism of a free group, and that the map $S(f_\lambda)$ is in fact ergodic. Unfortunately, the details on these two points are sparse in Thurston's paper and we remedy this as described in the next paragraph. Finally, to extend the results to \emph{weak} Perron numbers, Thurston uses the fact that for a weak Perron number $\lambda$, there exists $N \in \mathbb{Z}^+$ such that $\lambda^N$ is Perron. 

The main contributions of this expository paper are as follows. First and foremost, we distill and streamline all of the pieces needed for Thurston's traintrack theorem from \cite{Thurston2014Entropy}, which contains a variety of additional theorems regarding interval maps and other topics. Second, we correct the number theoretic error in the Even Lemma mentioned above. Third, we add a proof of the ergodicity of the traintrack maps, which was completely missing in Thurston's paper. Finally, we use Stallings folds to thoroughly prove that the traintrack maps represent elements of $\Out(\mathbb{F}_n)$. For this last piece, Thurston outlines an algebraic argument, but provides only a few details. We found the argument using Stallings folds more straightforward and rigorous. 

We conclude with some remarks. First, the notion of traintrack maps for $\Out(\mathbb{F}_n)$ is motivated by Thurston's traintracks on surfaces. In particular, elements of $\Out(\mathbb{F}_n)$ that admit traintrack representatives are the analog of pseudo-Anosov homeomorphisms of surfaces, and expansion constants for these maps are the analog of stretch factors for pseudo-Anosovs. Despite the fact that traintrack theory for free groups is more complicated than for surfaces, Thurston gave a complete answer for which algebraic integers can arise as expansion constants for traintrack representatives of outer automorphisms. Thurston also showed that, for surfaces, every stretch factor of a pseudo-Anosov homeomorphism is an algebraic unit, but it is still unknown exactly which units can appear as stretch factors. There has been partial progress in this direction. For example, 
Fried \cite[Theorem 1]{Fried1985Growth} proved every stretch factor $\lambda$ of a pseudo-Anosov homeomorphism is a bi-Perron unit. 
Fried further conjectured \cite[Problem 2]{Fried1985Growth} that every bi-Perron unit has some power such that it is realized as a stretch factor of a pseudo-Anosov homeomorphism. The work of Pankau \cite{Pankau2020Salem} and Lichiti--Pankau \cite{Liechti2021Geometry} made some progress towards the conjecture, but the question is still far from resolved. 

In the surface case, the stretch factors for a pseudo-Anosov homeomorphism and its inverse are always equal. This is no longer the case for outer automorphisms of $\F_{n}$. See \cite[Remark p.9]{Bestvina1992train} for an example of this. In \cite{Thurston2014Entropy}, Thurston asks the question: Which pairs of numbers can appear as the expansion constants of an outer automorphism and its inverse? He gives a partial answer when one restricts to the class of \emph{bipositive} outer automorphisms, but he notes that his result cannot generalize to all outer automorphisms. He also conjectures that every pair of weak Perron numbers greater than 1 is a pair of expansion constants for an outer automorphism and its inverse. We are not aware of a proof of this claim in the literature as of this moment, but a proof of this claim would be a good first step towards answering Thurston's question. For the sake of brevity we do not discuss this portion of Thurston's paper and direct interested readers to \cite[Section 11 \& 12]{Thurston2014Entropy}.

\subsection{Outline of paper} In \Cref{sec:Background} we provide all revelent background on entropy, $\lambda$-expanding graph maps, traintrack structures, Stallings folds, and Perron numbers, in that order. In \Cref{sec:Asterisk_Maps}, we define star maps and use the geometry of Perron numbers to construct uniformly $\lambda$-expanding star maps for all Perron and weak Perron numbers $\lambda$. In \Cref{sec:Prototypes}, we define Thurston's prototype traintrack maps $\phi_n$ and prove that they are indeed homotopy equivalences so that they represent elements of $\Out(\mathbb{F}_n)$. Then, in \Cref{sec:Splitting}, we define split graphs, split maps, and prove that our split star maps are traintrack representatives of elements of $\Out(\mathbb{F}_n)$. Finally, in \Cref{sec:Conclusion}, we prove that the split star maps are ergodic and finish the proof of \Cref{thm:thurston_main}.

\section*{Acknowledgements}
The authors would like to thank Jack Cook,  Sean Howe, Michael Kopreski, and Kurt Vinhage for many helpful discussions and thank Rylee Lyman for helpful comments on the earlier draft of this paper. We would also like to thank Mladen Bestvina for his suggestion that we undertake this project during a yearlong seminar supported by the NSF RTG Grant DMS–1840190. We also thank the referee for numerous insightful comments and suggestions.

The authors acknowledge support from NSF grants RTG DMS--1745583 (Dickmann), DMS--1905720 (Domat, Kwak), RTG DMS--1840190 (Domat, Rechkin), CAREER DMS--2046889 (Kwak, Patel), CAREER DMS--1452762 (Ospina), and the University of Utah Faculty Fellows Award (Patel). 

\tableofcontents

\section{Background}\label{sec:Background}\vspace{1em}

Thurston uses a variety of tools from dynamics, geometric group theory, and algrebraic number theory throughout his proof of \Cref{thm:thurston_main}. All relevant background material on these three topics are given in this section. 
    \subsection{Entropy}\label{ssec:BG_Entropy}
        The value of the topological entropy describes the complexity of the orbit structure of a dynamical system, and is an invariant of topological conjugacy classes. 
Let $(X,d)$ be a compact metric space and let $f:X \to X$ be a continuous map.
For $n \geq 1$, set
\[
d_n(x,y) := \max_{0\leq j \leq n-1} d(f^j(x),f^j(y)).
\]
Denote by $B_\varepsilon^n(x)$ the open $\varepsilon$-ball $\{y \in X\ |\ d_n(x,y)<\varepsilon\}$ with respect to $d_n$.

Since $X$ is a compact space, for each $\varepsilon>0$, $X$ can be covered by a finite collection of open sets of the form $B_\varepsilon^n(x_i)$ for $x_i \in X$. Then, let $N(\varepsilon,n)$ be the minimum number of such open sets that cover $X$.
\begin{DEF}
Let $f:X \to X$ be a continuous map on a compact metric space $X$. The \textbf{topological entropy} $h(f)$ is 
\begin{equation*}
    h(f) = \lim_{\varepsilon \to 0} \limsup_{n \to \infty} \frac{1}{n} \log(N(\varepsilon,n)).
\end{equation*}
\end{DEF}

\subsection{Graph Maps and Entropy}\label{ssec:BG_Expansion_Constants}
Now we turn our attention to the category of \emph{graphs}.
A \textbf{graph} $\G$ is a 1-dimensional CW complex, whose 0-simplices are called the \textbf{vertices}, and whose 1-simplices are called the \textbf{edges}.
We will always assume our graphs have finitely many vertices and edges.
If a graph is given an orientation on each of the edges (a choice of left and right endpoints), then it is called an \textbf{directed graph}.

A \textbf{graph map} $f:\G_{1} \to \G_{2}$ between graphs is a continuous map that sends vertices to vertices and edges to edge paths. 

Endowing the edges of a graph with lengths produces a \textbf{metric graph}. With a metric, the length of any edge path can be measured as the sum of lengths of edges in the path. Namely, this induces a \textbf{length function} $\ell$ on the set of edge paths in the graph. A typical choice of such a metric is the \textbf{combinatorial metric}: the metric that gives every edge length 1.

For a given metric graph $\G$ with a length function $\ell$, 
we define the \textbf{total variation} of a graph map $f:\G \to \G$ as: 
\[
    \Var(f) =\sum_{e \in E(\G)}\ell(f(e)),
\] 
where $E(\G)$ denotes the set of edges of $\G$. 

Computing the topological entropy for a graph map is simple due to the 1-dimensionality of the graph, and we will use the following theorem to compute the entropy of graph maps throughout the paper.
\begin{THM}[{\cite{Aldeda2000Combinatorial}}]
    \label{thm:entropyOfGraphMaps}
    Let $f: \G \to \G$ be a graph map on a finite metric graph that has finitely many points at which $f$ is not a local homeomorphism. Then
    \[
        h(f) = \lim_{n \to \infty} \frac{1}{n} \log(\Var(f^n)).
    \]
\end{THM}

\subsection{Perron-Frobenius Theorem}\label{ssec:BG_PF_Theorem}

Thurston's argument uses the following form of the well-known Perron-Frobenius theorem. We say a matrix is nonnegative (or positive) when its entries are nonnegative (or positive, respectively). A nonnegative matrix is said to be \textbf{ergodic} if the sum of a finite number of its consecutive positive powers is a positive matrix.
A nonnegative matrix is said to be \textbf{mixing} if some power is a positive matrix. 
These two notions describe that the dynamics on a space will not decompose into subspaces that have independent dynamics. Mixing is stronger than ergodicity; a mixing map exhibits more uniform orbit dynamics. We note that Thurston uses the definition of ergodicity above. 
However, in the literature, this definition of ergodicity is commonly referred to as \emph{irreducible}. What we defined as mixing is also commonly referred to as \emph{primitive}, see for instance \cite{LindMarcus2021}, definitions 4.2.2 and 4.5.7.

\begin{THM}[Perron-Frobenius] \label{thm:PF_Theorem}
Let $M$ be an $n \times n$ matrix with integer entries. If $M$ is nonnegative, then $M$ has at least one eigenvector such that

\begin{enumerate}[(i)]
    \item The corresponding eigenvalue $\lambda$ is nonnegative.
    \item $\lambda \geq | \lambda_i|$ for all the other eigenvalues $\lambda_i$.
\end{enumerate}

Furthermore, if $M$ is ergodic then there is a unique eigenvector whose corresponding eigenvalue $\lambda$ is strictly positive.
\end{THM}

To use the Perron-Frobenius theorem in our setting, we relate a (self) graph map $f$ with the \emph{transition matrix} $M$, defined as follows. Enumerate the edges of $\G$ by positive integers, say $1,\ldots,n$. Then the \textbf{transition matrix} of $f:\G \to \G$ is an $n\times n$ integer matrix, whose $(i,j)$ entry is determined by the number of time the $i$-th edges appear in the edge path $f(j)$. In light of the definitions of ergodic and mixing matrices, we call a graph map \textbf{ergodic} (or \textbf{mixing}) if its transition matrix is ergodic (or mixing, respectively).

Perron-Frobenius implies that an ergodic graph map $f:\G \to \G$ has a transition matrix with a \emph{positive} leading eigenvalue $\lambda$ that is no smaller than the magnitude of the other eigenvalues. Therefore, $\Var(f^n) \sim \lambda^n$ and it follows that $h(f) = \lim_{n \to \infty} \frac{1}{n}\log(\Var(f^n)) = \log(\lambda)$ by \Cref{thm:entropyOfGraphMaps}. This proves the forward direction of \Cref{thm:thurston_main} since, as we will discuss in \Cref{ssec:BG_Traintrack_Structures}, a traintrack representative of an outer automorphism of a free group is a special case of a graph map.

    \subsection{Uniform $\lambda$-Expanders}\label{ssec:BG_Uniform_Expanders}

\begin{DEF}\label{def:expander}
Let $\lambda>0$ be an algebraic integer. Let $\Gamma$ be a metric graph and $\ell$ be a length function on the edge paths in $\Gamma$. We say $f$ is \textbf{uniformly $\lambda$-expanding} if every edge is scaled by the same factor $\lambda$; namely, $\ell(f(e)) = \lambda \ell(e)$ for every edge $e$ of $\Gamma$.
\end{DEF}

Throughout the paper, we construct many maps of this form while working towards the main theorem due to the following useful proposition. 

\begin{PROP}
    \label{prop:entropyOfUniformMaps}
    Let $\lambda$ be an algebraic integer and $f$ be a uniformly $\lambda$-expanding graph map on a finite metric graph. Then $h(f) = \log \lambda$.
\end{PROP}

\begin{proof}
    Let $f: \G \to \G$ and 
    $\ell$ denote the length function on $\G$.
    Observe that $\Var(f) = \lambda \ell(\G)$, and more generally $\Var(f^n) = \lambda^n \ell(\G)$. Using \Cref{thm:entropyOfGraphMaps}, we compute
    \[
        h(f) = \lim_{n \to \infty} \frac{1}{n}\log (\Var(f^n)) =  \log(\lambda) + \lim_{n \to \infty } \frac{1}{n}\log(\ell(\G))= \log(\lambda).
        \qedhere
    \]
\end{proof}

    \subsection{Traintrack Structures}\label{ssec:BG_Traintrack_Structures}
   
    Now we will describe the necessary background required to understand the statement of \Cref{thm:thurston_main}. The celebrated work of Bestvina and Handel \cite{Bestvina1992train} gives a geometric approach to understanding $\Out(\F_n)$, the outer automorphism group of a finitely generated free group $\F_n$. See \cite{Bestvina2002, OHGGT} and for a general introduction to $\Out(\F_{n})$ and traintrack maps. Recall
    \[
      \Out(G) = \Aut(G) / \Inn(G),
    \]
where $\Aut(G)$ is the group of automorphisms and $\Inn(G)$ is the group of inner automorphisms of $G$, i.e., the subgroup of $\Aut(G)$ consisting of automorphisms of the type $\varphi: x \mapsto gxg^{-1}$ for $g \in G$.

Bestvina and Handel \cite[Theorem 1.7]{Bestvina1992train} showed that certain outer automorphism 
can be represented by a special map between graphs called a traintrack map (see also \cite[Chapter 6.3]{OHGGT}), which we now define. 

A graph map is \textbf{taut} if it restricts to a local embedding on the interior of each edge. The term taut was chosen to represent the fact that all backtracking in the graph map has been removed. Indeed, any graph map is homotopic to a taut graph map.

A self-graph map $f: \Gamma \to \Gamma$ that is also a homotopy equivalence induces an automorphism of the fundamental group, which is well-defined up to postcomposing an inner automorphism depending on the choice of path connecting base points. The fundamental group of a graph is $\F_n$ for some $n$, so the map $f$ corresponds an element $\psi \in \Out(\F_n)$. The notion of tautness of a graph map is analogous to cyclically reduced words in free groups.

 The standard topology on the intervals descends to a topology on the graph. Consider a small closed neighborhood of a vertex in an directed graph, which we can view as an directed graph itself. The oriented edges in this neighborhood are called \textbf{directions}. Formally, a \textbf{turn} in the graph is a 2-element subset of directions from a single vertex. More intuitively, we can think of a turn as a segment of a path passing through the vertex.
In an oriented graph, we can refer to a turn with a 2-letter long word, where an uppercase letter represents traversing an edge backwards. For example, the turn $bA$ will refer to the turn a path makes after first traveling along the edge $b$ in the forward direction followed by $a$ in the reverse direction. Note that reversing the order of the word and swapping the case of the letters gives the same turn, but now the path representing the turn is traveling in the opposite direction; e.g., $aB$ is the same turn as $bA$. 
In either case, we say that the path takes the given turn. 
In general, a locally embedded path takes a turn at every vertex it crosses, and the turn is determined by the incoming and outgoing edges the path takes at a given vertex. 
 
Now we partition the set of turns into two collections, a set of legal turns and a set of illegal turns. We will always assume that every \emph{backtracking}, a turn of the form $xX$, is illegal. Such a partition is called a \textbf{traintrack structure} on $\G$. A path is \textbf{legal} if it is a local embedding, and it takes a legal turn at each vertex on the path. A path is \textbf{illegal} if it is not legal. Note a graph map sends a turn to another turn, so we have the following definition.

\begin{DEF}
A graph map $f: \Gamma \to \Gamma$ is a \textbf{traintrack map} when it is taut and there exists a traintrack structure on $\Gamma$ such that

\begin{enumerate}[(i)]
    \item legal turns are sent to legal turns, and
    \item every edge is sent to a legal path.
\end{enumerate}

\end{DEF}

The conditions (i) and (ii) together imply that legal paths are sent to legal paths under a traintrack map. Thurston defines a traintrack map to be a map of a graph such that all iterates are local embeddings on each edge. The fact that legal paths are sent to legal paths shows that our definition implies Thurston's.  A traintrack map $f:\Gamma \to \Gamma$ is called \textbf{irreducible} if $f$ does not admit an invariant proper subgraph that is not a tree.

\begin{figure}[ht!]
	    \centering
	    \def\svgwidth{\columnwidth}
		    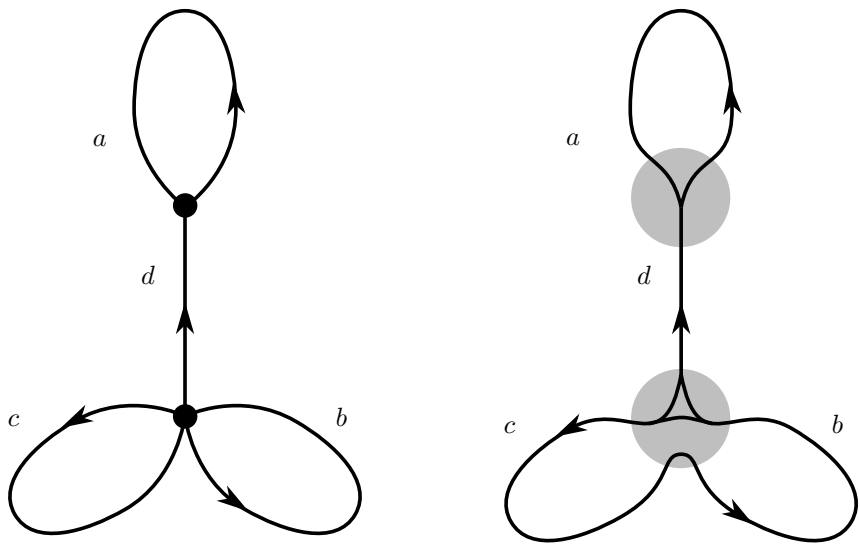
		    \caption{Left: A graph with labeled oriented edges. Right: A traintrack structure on the graph  with legal turns $da$, $dA$, $Dc$, $DB$, $bc$, and $cb$.}
	    \label{fig:extrain}
\end{figure}

A traintrack structure on a graph is often graphically represented by blowing up each vertex to a disk. Then for each legal turn we draw a smooth path within the corresponding disk which connects the endpoints of edges from the legal turn. The legal turns then correspond to turns which an actual train could make traveling along a track modeled after the picture, i.e., avoiding sharp turns. See \Cref{fig:extrain} for an example. In the given traintrack structure from \Cref{fig:extrain}, no legal path can travel along the $a$ edge in either direction twice in succession (the turn $aa$ is illegal), nor can it travel along $d$ immediately after traveling along $c$ in the forward direction (the turn $cd$ is illegal), etc.

\subsection{Stallings Folds}\label{ssec:BG_Stallings_Folds}
In \Cref{sec:Prototypes} and \Cref{sec:Splitting},
we will make use of Stallings folds, first introduced in \cite{Stallings1983}, to verify that our graph maps are homotopy equivalences. Note that this differs from Thurston's approach in the original proof of \Cref{lem:prototypehe}. We felt that using Stallings folds was more intuitive for a rigorous proof.

\begin{DEF}
    A \textbf{morphism} of graphs is a continuous map of graphs that sends vertices to vertices and edges to edges. An \textbf{immersion} of graphs is a morphism that is locally injective.
\end{DEF}

 Note that a morphism of graphs is different from a graph map since for graphs maps, we require only that edges be sent to \emph{edge-paths}.
Beginning with a graph map, we can subdivide the edges of the domain graph at the complete pre-image of the vertices of the codomain graph in order to obtain a graph morphism in the style of Stallings. 
Note that the property of a morphism being an immersion needs to only be checked at the vertices of the domain. 

\begin{DEF}
    Let $\Gamma$ be a graph and $x_{1}$, $x_{2}$ be two edges of $\Gamma$ sharing a vertex. Let $\Gamma' = \Gamma / (x_{1} \sim x_{2})$. A \textbf{fold} is the natural quotient morphism $\Gamma \rightarrow \Gamma'$.
\end{DEF}

Folds come in two flavors depending on whether the two edges $x_{1},x_{2}$ share a single vertex or share both vertices. Folds between edges that share only a single vertex are called \textbf{Type I} folds and are homotopy equivalences. Folds between edges that share both vertices are \textbf{Type II} folds and fail to be homotopy equivalences (they reduce the rank of the fundamental group by one). See \Cref{fig:foldtypes} for examples of Type I and Type II folds. 

\begin{figure}[ht!]
	    \centering
	    \def\svgwidth{.8\columnwidth}
		    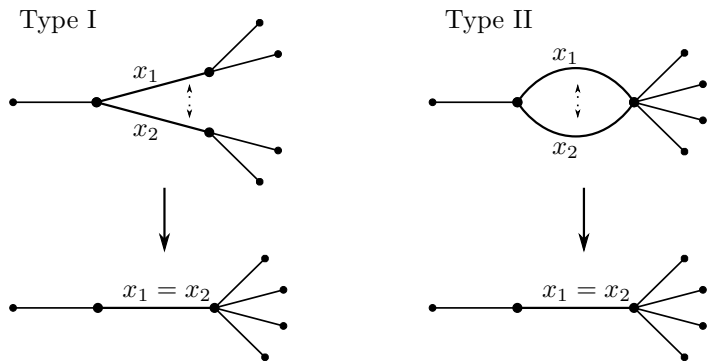
		    \caption{Examples of Type I and Type II folds.}
	    \label{fig:foldtypes}
\end{figure}

\begin{THM}[\cite{Stallings1983}]\label{thm:stallings}
    Let $\phi:\Gamma \rightarrow \Delta$ be a graph morphism. Then $\phi$ factors as 
    \begin{align*}
        \Gamma = \Gamma_{0} \xrightarrow{F_{1}} \Gamma_{1} \xrightarrow{F_{2}} \cdots \xrightarrow{F_{n}} \Gamma_{n} \xrightarrow{\psi} \Delta,
    \end{align*}
    where each of the $F_{i}$ are folds and $\psi$ is an immersion. 
\end{THM}

This theorem gives an algorithm for verifying that a given graph map is a homotopy equivalence: First one subdivides the domain in order to obtain a graph morphism. Then one performs all possible folds. If all of the folds performed are Type I folds and the final immersion, $\psi$, is a homeomorphism(or equivalently, graph isomorphism) then the original map is a homotopy equivalence. In fact, for our applications we will always only perform Type I folds and $\psi$ will be a graph automorphism.

    \subsection{Perron Numbers}\label{ssec:BG_Perron_Numbers}
        In this section, we compile relevant definitions and facts in algebraic number theory, toward the introduction of Perron numbers.
    
    \begin{DEF}
    An \textbf{algebraic integer} is a complex number that is a root of some monic polynomial with integer coefficients. A polynomial with integer coefficients is also called an integer polynomial.
    
    Given an algebraic integer $\alpha$, the \textbf{minimal polynomial} $p_\alpha$ of $\alpha$ is the integer monic polynomial of the least degree that has $\alpha$ as a root. Then the \textbf{degree} of $\alpha$ denoted by $\deg \alpha$, is the degree of its minimal polynomial $p_\alpha$.
    The \textbf{Galois conjugates}, sometimes simply called the conjugates, of $\alpha$ are the other roots of $p_\alpha$.
    \end{DEF}
    
    \begin{DEF}
    A \textbf{weak Perron number} is a real algebraic integer
    $\lambda=\lambda_1$, whose Galois conjugates $\lambda_2,\ldots,\lambda_d$ have modulus no larger than $\lambda$:
    \[
        \lambda \ge |\lambda_i|, \qquad \text{ for all }i=1,\ldots,d.
    \]
    
    A \textbf{(strong) Perron number} is a real algebraic integer $\lambda=\lambda_1$ that satisfies the strict inequality for the moduli of Galois conjugates $\lambda_2,\ldots,\lambda_d$, namely:
    \[
        \lambda > |\lambda_i|, \qquad \text{ for all }i=2,\ldots,d.
    \]
    \end{DEF}
    
    The following fact will be used in \Cref{sec:Conclusion} to expand our result on Perron numbers to weak Perron numbers.
    \begin{PROP}[{\cite[Theorem 3]{Lind1984Entropy}}]
    \label{prop:PnWP}
        Let $\lambda$ be an algebraic integer. Then $\lambda$ is weak Perron if and only if $\lambda^n$ is Perron for some positive integer $n$.
    \end{PROP} 
    
    Now let $\lambda$ be a Perron number.
    Denote by $\Q(\lambda)$ the \textbf{number field} that is the smallest field extension over $\Q$ containing $\lambda$. As $\lambda$ is an algebraic integer, we have $\Q(\lambda) = \Q[x]/(p_\lambda)$. Note $\deg \Q(\lambda) = \deg \lambda$. Let $\cO_\lambda$ be \textbf{the ring of integers} in the field $\Q(\lambda)$, defined as the set of all algebraic integers in $\Q(\lambda)$. Then the ring of integers will be realized as a submodule of the number field with exactly $\deg \lambda$ basis elements:
    \begin{FACT}[{\cite[Theorem 9, pp.20--21]{Markus2018Number}}]\label{fact:ROIsublattice}
        If $\mathbb{Q}(\lambda)$ is a number field of degree $d$, then its ring of integers $\cO_\lambda$ is a free $\Z$-submodule of $\mathbb{Q}(\lambda)$ of rank $d$. 
        In other words, there are $d$ elements $\alpha_1,\ldots,\alpha_d \in \cO_\lambda$ 
        such that
        \[
            \cO_\lambda = \{m_1\alpha_1+\ldots+m_d\alpha_d\ |\ m_1,\ldots,m_d \in \Z\}.
        \]
    \end{FACT}
    
    \begin{RMK}\label{rmk:ROImaximallattice}
        
        For an algebraic integer $\mu$ of degree $d$, $\Z[\mu] \subset \cO_\mu$ will also be a rank $d$ submodule of $\Q(\mu)$, but $\Z[\mu]$ may not be \emph{maximal} in the sense that, for some $\mu$, the containment is proper $\Z[\mu] \subsetneq \cO_\mu$.
        
        For example, when $\mu = \sqrt{5}$, then $\cO_\mu = \Z\left[\frac{1+\sqrt{5}}{2}\right]$ as $5 \equiv 1 \mod 4$. (refer to \cite[Chapter 2, Corollary 2]{Markus2018Number}). This properly contains $\Z[\sqrt{5}]$ and both of these are rank-2 submodules of $\Q[\sqrt{5}]$.
        Even when $\lambda$ is a Perron number, for example take $\lambda = 1+\sqrt{5}$, it is possible that $\Z[\lambda] \subsetneq \cO_\lambda$. Additionally, there are many examples where $\cO_\lambda \neq \Z[\mu]$ for every algebraic integer $\mu \in \cO_\lambda$. (For an example, refer to \cite[pp 64--65]{Narkiewicz2004Elementary}.) Such are said to be \textbf{non-monogenic}.
    \end{RMK}
    
    For \Cref{ssec:AST_Uniform_Expanding}, we need the following lemma. 
    \begin{LEM}[Even Lemma] \label{lem:even}
       Let $\lambda$ be an algebraic integer. Then, there exist $n \neq n_0 \in \Z^+$ such that
        \[
            \lambda^n \equiv \lambda^{n_0} \mod 2\cO_\lambda.
        \]
    \end{LEM}
    \begin{proof}
    Since $\cO_\lambda/2\cO_\lambda \cong (\Z/2 \Z)^d$ is finite, by the pigeonhole principle there must be some $n \neq n_0 \in \Z^+$ for which $\lambda^n \equiv \lambda^{n_0} \mod 2\cO_\lambda$. 
   
    \end{proof}
    
    \begin{RMK}[Counterexample for $n_0 = 0$]\label{rmk:n0error}
    In \cite{Thurston2014Entropy}, Thurston stated the Even Lemma with $n_0=0$, but that statement is false in general. Indeed, consider $\lambda = \frac{3+\sqrt{17}}{2}$.  This is a degree 2 Perron number with minimal polynomial $p_\lambda(x) = x^2 - 3x - 2$.  Because $17 \equiv 1 \mod 4$, $\mathcal{O}_\lambda = \Z[\lambda]$. We can write the basis elements $1, \lambda \in \Z[\lambda]=\mathcal{O}_\lambda$ as ordered pairs with respect to the basis for $\Q(\lambda) = \Q^2 = \langle 1 , \lambda \rangle$; that is $\Z[\lambda] = \langle 1, \lambda \rangle = \langle (1, 0), (0, 1) \rangle$.  
    
    We claim that there is no $n \in \Z^+$ for which $\lambda^n \equiv 1 = (1, 0) \mod 2\cO_\lambda$,
    where here $2\cO_\lambda = \langle (2, 0), (0, 2) \rangle$.  In fact, we assert that $\lambda^n \equiv (0, 1) \mod2\cO_\lambda$ for all $n \in \Z^+$. Using the relationship given from the minimal polynomial $\lambda^2 = 3\lambda + 2$, we see that this is true for the first few powers of $\lambda$: 
    \begin{align*} 
        \lambda^1 &= (0,1) \not\equiv (1, 0) \mod 2\cO_\lambda,\\ 
        \lambda^2 &= (2, 3) \equiv(0,1) \not\equiv (1, 0) \mod 2\cO_\lambda,\\ 
        \lambda^3 &= (6, 11) \equiv(0,1) \not\equiv (1, 0) \mod 2\cO_\lambda.\\
    \end{align*} 
    More generally, as we will see in \Cref{lem:CompanionMatrixMult}, the multiplication by $\lambda$ on an element in $\Q(\lambda) \supset \cO_\lambda$ can be realized as the multiplication by the companion matrix $C_\lambda = \begin{pmatrix}
    0 & 2\\
    1 & 3
    \end{pmatrix}$ on the corresponding ordered pair written as a column vector.  It follows inductively that for $n >1$:
    \begin{equation*} 
    \lambda^n = C_\lambda \cdot  \lambda^{n - 1} \equiv \begin{pmatrix}
    0 & 2\\
    1 & 3
    \end{pmatrix}\begin{pmatrix} 0\\1 \end{pmatrix} = \begin{pmatrix} 2\\3 \end{pmatrix}
    \equiv
    \begin{pmatrix}
    0 \\ 1
    \end{pmatrix}
    \mod 2\cO_\lambda,
    \end{equation*} 
    thus $\lambda^n \equiv (0, 1) \not\equiv(1, 0) \mod 2\cO_\lambda$ for all $n \in \Z^+$.  
    
    \end{RMK}

\section{Star Maps}\label{sec:Asterisk_Maps}\vspace{1em}
    \subsection{Definition of a Star Map}\label{ssec:AST_Definition}

\begin{DEF} A \textbf{star} (referred to as an asterisk graph in \cite{Thurston2014Entropy}) with $n$-tips is the complete bipartite graph $\ast_n = K_{1, n}$. 

\begin{center}
    \asterisk{3}{1} \qquad \asterisk{4}{1} \qquad \asterisk{5}{1} \qquad
\begin{tikzpicture}
    \node[circle,fill=red,scale=0.3] at (360:0mm) (center) {};
    \foreach \n in {1,...,7}{
        \node[circle,fill=blue,scale=0.3] at ({(\n-3)*360/9}:1cm) (n\n) {};
        \draw (center)--(n\n);
    }
    \node at (-.5,-.25) {$\ddots$};
    \node at (0,-1*1.5) {$\ast_n$}; 
\end{tikzpicture}
\end{center}

A \textbf{star map} is a self graph map of a star, $f \colon \ast_n \to \ast_n$ such that: 
\begin{itemize} 
\item[(a)] $f$ fixes the center vertex (in this way $f$ preserves the bipartite structure of $\ast_n$), and
\item[(b)] the map that sends each edge to the first edge of its image is a permutation. We will refer to this map as the \textbf{first edge map}.
\end{itemize} 
\end{DEF}

We note that a star map can be always homotoped to be taut. Namely, the image of each edge can be reduced by canceling back-tracking, only leaving the last letter.
However, we will let our star maps have back-trackings, and will not pass to a simpler representative in its homotopy class. Such redundancy will be crucial when we are defining the \textit{split} of a star map in \Cref{sec:Splitting}.

\begin{EX} Let $f \colon \ast_3\to \ast_3$ be the star map defined by: 

\begin{center}
\begin{tabular}{m{3cm} m{3cm}}
\\
\begin{tikzpicture}
    \node[circle,fill=red,scale=0.3] at (360:0mm) (center) {};
    \foreach \n in {1,...,3}{
        \node[circle,fill=blue,scale=0.3] at ({(\n)*360/3}:1cm) (n\n) {};
        \draw[midarrow]  (center)--(n\n);
        }
     \node at ({360/3-20}:0.5cm) {$a$};
     \node at ({2*360/3-20}:0.5cm) {$b$};
     \node at ({3*360/3-20}:0.5cm) {$c$};
\end{tikzpicture}

&
$
f : \begin{cases} 
a \mapsto cCb\\
b \mapsto bBa\\
c \mapsto aAbBc
\end{cases} 
$

\end{tabular}
\end{center}
Since $f$ maps each edge in $\ast_3$ to an edge path of odd (unsigned) length it must fix the center vertex.  Also, letting $f_1$ be the first edge map, we have that $\text{Im}(f_1) = \{a,b,c\}$.  Thus, $f$ is a star map.  
\end{EX}

The aim of this section is to prove the existence of a uniformly $\lambda$-expanding star map for every Perron number $\lambda$.  

\begin{THM}[{\cite[Theorem 6.2]{Thurston2014Entropy}}, Uniformly Expanding Star Maps for Perron Numbers]
\label{thm:astmapPerron}
    Let $\lambda$ be a Perron number. Then there exists $n > 0$ and a star map $f: \ast_n \to \ast_n$ such that $f$ is a uniform $\lambda$-expander with mixing transition matrix.
    
\end{THM}

\begin{EX}
\label{ex:5-uniformAstExample}
The following is an example of a star map that satisfies \Cref{thm:astmapPerron}.  
Suppose $\lambda = 5$. Consider the star graph $\ast_4$ and endow each edge with length 1. Let $f \colon \ast_4 \to \ast_4$ be defined by 

\begin{center}
\begin{tabular}{m{3cm} m{3cm}}
\begin{tikzpicture}
    \node[circle,fill=red,scale=0.3] at (360:0mm) (center) {};
    \foreach \n in {1,...,4}{
        \node[circle,fill=blue,scale=0.3] at ({(\n)*360/4}:1cm) (n\n) {};
        \draw[midarrow]  (center)--(n\n);
        }
     \node at ({360/4-20}:0.5cm) {$a$};
     \node at ({2*360/4-25}:0.5cm) {$b$};
     \node at ({3*360/4-20}:0.5cm) {$c$};
     \node at ({4*360/4-20}:0.5cm) {$d$};
\end{tikzpicture}
&
$
f \colon \begin{cases} 
    a \mapsto bBdDb\\
    b \mapsto aAcCc\\
    c \mapsto dDbBa\\
    d \mapsto cCbBd
\end{cases} 
$
\end{tabular} 
\end{center} 
We see that $f$ is a star map since the first edge map $f_1$ permutes the edges ($a \to b \to a$ and $c \to d \to c$),

and $f$ respects the bipartite structure since it maps each edge to an edge path of odd length.  Furthermore, $f$ is a $5$-uniform expander since $\ell(f(e)) = 5 = 5 \ell(e)$ for all edges $e$.  

It is straighforward to check that the cube of the transition matrix is positive so $f$ is mixing. This example can easily be generalized to a construction for $\lambda$ equal to any \emph{odd} integer.

\end{EX}

    \subsection{Geometry of Perron Numbers}\label{ssec:AST_Geometry_of_Perron_Numbers}
        Let $\lambda$ be a Perron number with $\deg \lambda = d$. In this section, we will study the dynamics of $\lambda$-multiplication on $\mathbb{Q}(\lambda)$.

    Viewing $\Q(\lambda)$ as a $d$-dimensional $\Q$-vector space, we \emph{extend scalars} by tensoring $\Q(\lambda)$ with $\R$:
    \[
        V_\lambda:= \Q(\lambda) \otimes_\Q \R \cong \R^d,
    \]
    where the latter isomorphism comes from realizing $\Q(\lambda)$ as a $d$-dimensional $\Q$-vector space with basis $1,\lambda,\ldots,\lambda^{d-1}$. Then, we have an isomorphism:
    \[
        V_\lambda = \Q(\lambda) \otimes_\Q \R \longrightarrow \R^d, \qquad 
        \sum_{i=0}^{d-1} (\lambda^i \otimes_\Q r_i)    \longmapsto    (r_0,\ldots,r_{d-1}).
    \]
    
    Using this identification, we can view a number $q \in \Q(\lambda)$ as a $d$-dimensional column vector $\mathbf{q} \in V_\lambda$. Namely, with the basis $\{1,\lambda,\ldots,\lambda^{d-1}\}$ of $\Q(\lambda)$, write $q = \sum_{i=0}^{d-1} (\lambda^i \cdot q_i) \in \Q(\lambda)$ with $q_i \in \Q$ for all $i=0,\ldots,d-1$. Then $q$ corresponds to a $d$-dimensional column vector $\mathbf{q}=[q_0,\ldots,q_{d-1}]^t$ in $V_\lambda$.
    
    Now we convert the $\lambda$-multiplication action on $\Q(\lambda)$ into matrix multiplication by the \emph{companion matrix}, $C_\lambda$, of $\lambda$ on the corresponding vector space $V_\lambda \cong \mathbb{R}^d$.
    Recall the \textbf{companion matrix} $C_\mu$ of an algebraic integer $\mu$ of degree $d$ is the $d \times d$ matrix:
    \[
        C_\mu =
        \begin{pmatrix}
            0 & 0 & \cdots & 0 & -c_0 \\
            1 & 0 & \cdots & 0 & -c_1 \\
            0 & 1 & \cdots & 0 & -c_2 \\
            \vdots & \vdots & \ddots & \vdots & \vdots \\
            0 & 0 & \cdots & 1 & -c_{d-1}
        \end{pmatrix}
    \]
    where the $c_i$'s are the integer coefficients of the minimal polynomial of $\mu$; $p_\mu(x) = c_0+c_1x + \ldots + c_{d-1}x^{d-1}+x^d$.
    
    \begin{LEM}
    \label{lem:CompanionMatrixMult}
    Let $\lambda$ be an algebraic integer, and $q \in \Q(\lambda)$.
    Then $\lambda \cdot q \in \Q(\lambda)$ corresponds to $C_\lambda 
    \cdot \mathbf{q} \in V_\lambda$, where $C_\lambda$ is the companion matrix of $\lambda$.
    \end{LEM}
    
    \begin{proof}
    Let $\deg \lambda =d$ with the minimal polynomial $p_\lambda(x) = c_0 + c_1 x + \ldots+ c_{d-1}x^{d-1}+x^d$. Since $p_\lambda(\lambda)=0$, we have
    \[
        \lambda^d = \sum_{i=0}^{d-1} -c_i \lambda^{i}.
    \]
    Writing $q \in \Q(\lambda)$ as $\sum_{i=0}^{d-1}(\lambda^i \cdot q_i)$ for $q_i \in \Q$, we identify $q$ with $\mathbf{q}=[q_0,\ldots,q_{d-1}]^t$ in $V_\lambda$. Now
    \begin{align*}
        \lambda \cdot q &= \lambda \sum_{i=0}^{d-1} (\lambda^i \cdot q_i) = \left\{\sum_{i=0}^{d-2} (\lambda^{i+1} \cdot q_i)\right\} + \lambda^d \cdot q_{d-1} \\
        &= \left\{\sum_{i=1}^{d-1}\lambda^i \cdot q_{i-1}\right\} + \sum_{i=0}^{d-1}(-c_i\lambda^{i}) \cdot q_{d-1}\\
        &= -c_0 q_{d-1} \cdot \lambda^{0} + \sum_{i=1}^{d-1} \left(q_{i-1}-c_iq_{d-1}\right)\cdot \lambda^i,
    \end{align*}
    where the last line can be identified with the matrix product:
    \[
        \begin{pmatrix}
            0 & 0 & \cdots & 0 & -c_0 \\
            1 & 0 & \cdots & 0 & -c_1 \\
            0 & 1 & \cdots & 0 & -c_2 \\
            \vdots & \vdots & \ddots & \vdots & \vdots \\
            0 & 0 & \cdots & 1 & -c_{d-1}
        \end{pmatrix}
        \begin{pmatrix}
            q_0 \\ q_1 \\ q_2 \\ \vdots \\ q_{d-1}
        \end{pmatrix}
        = C_\lambda \mathbf{q},
    \]
    concluding the proof.
    \end{proof}
    
    Given \Cref{lem:CompanionMatrixMult}, we will interchangeably use $C_\lambda$-action and $\lambda$-action by identifying $\lambda \cdot q \in \Q(\lambda)$ and $C_\lambda \cdot \mathbf{q} \in V_\lambda$.

    Next, we find a subset of $V_\lambda$ that is invariant under $C_\lambda$-multiplication. 
    As this is a linear action, the natural choice of such an subset is an eigenspace of $C_\lambda$.
    However, as $\lambda$ is a Perron number, we can find a more useful invariant space $K_\lambda$ as follows.
    First, by definition of the companion matrix, the minimal polynomial of $C_\lambda$ is exactly the minimal polynomial $p_\lambda$ of $\lambda$. Since any finite extension over $\Q$ is separable, $p_\lambda$ has no repeating roots.
    Hence, we can enumerate the eigenvectors of $C_\lambda$ as $v_1,\ldots,v_d$, where $v_1$ is associated with the leading eigenvalue $\lambda$. Normalize $v_1,\ldots,v_d$ to have norm 1. Each $\mathbf{v} \in V_\lambda$ can be expressed as $\mathbf{v}=a_1v_1+\ldots+a_dv_d$ for some $a_i \in \R$. Now define the invariant open cone $K_\lambda$ as follows:
    \[
        K_\lambda:= \{a_1v_1+\ldots+a_dv_d \in V_\lambda\ |\ a_1>0, \quad a_1>|a_i|, \text{ for all $i=2,\ldots,d$}\}.
    \]
    Namely, $K_\lambda$ is the set of all points in $V_\lambda$ whose projection to the $\lambda$-eigenspace of $C_\lambda$
    is positive and larger than the size of projection to any of the other eigenspaces.
    
    Note $K_\lambda$ is \textbf{polyhedral}, namely the cone is generated by finitely many vectors in $V_\lambda$. Indeed, the $2(d-1)$ bisectors, $\{v_{1} \pm v_{i}\}_{i=2}^{d}$, will generate $K_{\lambda}$. We label these by $\{w_1, \ldots, w_{2d-2}\}$ in the remainder of this section. Note that on each of these bisectors the projections onto either $v_{1}$ or $v_{i}$ have the same magnitude. See \Cref{fig:InvCone} for $K_\lambda$ with $d=2$.

    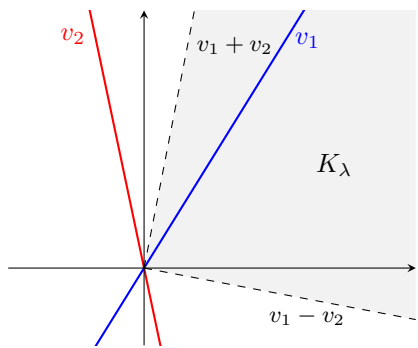
\begin{figure}[ht!]
    \begin{tikzpicture}
\begin{axis}[
    axis on top,
    axis lines = center,
    ticks=none,
    ymax=10,
    ymin=-3,
    xmax=10,
    xmin=-5,
]

\addplot [name path=v,
    domain=-20:20, 
    samples=100, 
    color=red,
    style=thick
    ]
    {-5*x} node[left,pos=.455] {$v_2$};

\addplot [name path=u,
    domain=-20:20, 
    samples=100, 
    color=blue,
    style=thick
    ]
    {17/10*x} node[right,pos=.63] {$v_1$};

\addplot [name path=A,
    domain=0:10, 
    samples=100, 
    color=black,
    style=dashed
    ]
    {-1/5*x} node[below,pos=.6] {\small $v_1 - v_2$};

\addplot [name path=B,
    domain=0:20, 
    samples=100, 
    color=black,
    style=dashed
    ]
    {27/5*x} node[right,pos=.08] {\small $v_1 + v_2$};
    
\addplot[gray!10] fill between[of=A and B, soft clip={domain=0:10},];

\node at (axis cs:7,4) {$K_\lambda$};

\end{axis}
\end{tikzpicture}
\caption{$v_1\pm v_2$ generate $K_\lambda$ when $d = 2$.} \label{fig:InvCone}
    \end{figure}

    The fact that $\lambda$ is Perron makes $K_\lambda$ invariant under $C_\lambda$ multiplication. Moreover, we have the proper containment $C_\lambda \cdot K_\lambda \subsetneq K_\lambda$. To see why, let $\mathbf{v}=a_1v_1+\ldots+a_dv_d$, with $\lambda_i$ being the associated eigenvalue corresponding to the eigenvector $v_i$ for $i=1,\ldots,d$. Note $\lambda_1=\lambda$. Then,
    \[
        C_\lambda \cdot \mathbf{v} = a_1\lambda_1v_1 + \ldots + a_d\lambda_d v_d.
    \]
    Because $\lambda$ is Perron, we have that $|\lambda_1|>|\lambda_i|$ for all $i=2,\ldots,d$. Therefore,
    \[
        |a_1\lambda_1| > |a_i\lambda_i|, \qquad \text{for all $i=2,\ldots,d$},
    \]
    showing that $C_\lambda \cdot \mathbf{v} \subset K_\lambda$. To see that the containment is proper, it suffices to show that the faces of $K_\lambda$ get mapped to the interior of $K_\lambda$.
    Indeed, let $\mathbf{v}=av_1+av_i$ with $a \neq 0$ be a point on a face of $K_\lambda$. Then $C_\lambda \cdot \mathbf{v}=a\lambda_1v_1+a\lambda_iv_i$. Since $|a\lambda_1| > |a\lambda_i|$, it follows that $C_\lambda \cdot \mathbf{v}$ has \emph{strictly} larger projection on $\<v_1\>$ than on $\<v_i\>$, showing $C_\lambda \cdot \mathbf{v}
    \in K_\lambda$.

    Recall by \Cref{fact:ROIsublattice}, we can embed $\cO_\lambda$ into $\Q(\lambda) \subset V_\lambda$ as a rank-$d$ lattice.
    In fact, we would like to trim the invariant cone $K_\lambda$ to have faces passing through such nice lattice points. Any such cone generated by vectors in the lattice will be called \textbf{rational}.
    
    \begin{PROP}[Rational Cone; cf. {\cite[Proposition 3.4]{Thurston2014Entropy}}]
    \label{prop:rationalCone}
        There is a rational polyhedral convex cone $KR_\lambda$ contained in $K_\lambda$ and containing $\lambda \cdot K_\lambda$.
    \end{PROP}
    
    \begin{proof}
       
        Consider the projective space $\P(V_\lambda) \cong \P(\R^d)$. 
        We first claim that $\P(\cO_\lambda)$ is dense in $\P(V_\lambda)$. This follows from the fact that $\mathbb{Q}(\lambda)$ is dense in $V_\lambda$ and that $\mathbb{P}(\cO_\lambda) = \mathbb{P}(\mathbb{Q}(\lambda))$.
        Indeed, by definition of the ring of integers, for any $x \in \Q(\lambda)$ there exists $m \in \Z^+$ such that $mx \in  \cO_\lambda$ where we can pick $m$ to be the least common multiple of the denominators of coefficients of the (monic and rational) minimal polynomial of $x$.
        
        Recall we have the proper containment $\lambda \cdot K_\lambda \subsetneq K_\lambda$. Let $w_1,\ldots,w_{2d-2}$
        be the generators of the cone $K_\lambda$ described above. Then $\lambda \cdot w_1,\ldots, \lambda \cdot w_{2d-2}$ are the generators of the cone $\lambda \cdot K_\lambda$ where $w_j \neq \lambda \cdot w_j$.
        Now, project the generators onto $\P(V_\lambda)$. 
        Note that $w_j$ and $\lambda \cdot w_j$ are \emph{not} identified in $\P(V_\lambda)$ because here $\lambda \cdot w_j$ corresponds to the vector $C_\lambda \cdot w_j$ and $w_j= v_1 \pm v_i$, where $v_1$ and $v_i$ are eigenvectors for different eigenvalues of $\C_\lambda$.
        
        All in all, $\P(\cO_\lambda)$ is dense in $\P(K_\lambda) \setminus \P(\lambda \cdot K_\lambda)$ which is nonempty, so for each $j=1,\ldots,2d-2$ we can pick $\overline{u_j} \in \P(\cO_\lambda)$ arbitrarily close to $\overline{w_j}$ such that the convex hull of $\{\overline{u_1},\ldots,\overline{u_{2d-2}}\}$ lies between $\P(\lambda \cdot K_\lambda)$ and $\P(K_\lambda)$.
    Finally, letting $KR_\lambda$ be the closed rational polyhedral cone generated by $u_1,\ldots,u_{2d-2}$, by construction we have    
        \[
            \lambda \cdot K_\lambda \subset KR_\lambda \subset K_\lambda,
        \]
    which concludes the proof.
    \end{proof}
    
    Define $S_\lambda \dfn \left(\cO_\lambda \cap KR_\lambda\right) \setminus \{0\}$, the set of lattice points in the rational cone $KR_\lambda$ in $V_\lambda$, minus the origin. Note $S_\lambda$ is equipped with a semigroup structure, as both $\cO_{\lambda}$ and $KR_\lambda$ are closed under the addition. Now we show it is also \emph{finitely generated}.
    
    \begin{PROP}[Gordan's Lemma; cf. {\cite[Proposition 3.5]{Thurston2014Entropy}}] \label{prop:Gordan}
    $S_\lambda$ is a finitely generated semigroup.
    \end{PROP}
    \begin{proof}
        Say $KR_{\lambda}$ is generated by $u_{1},\ldots,u_{k} \in \cO_{\lambda}$ constructed in the proof of \Cref{prop:rationalCone}. 
   Let $H:= \{\sum_{i=1}^{k}a_{i}u_{i}\, |\,  a_{i} \in [0,1]\} \subset KR_{\lambda}$. Since $H$ is compact and $\cO_{\lambda}$
    is discrete, it follows that $H \cap \cO_{\lambda}$ is finite.  We claim that this finite set, which we label $\{s_1, \ldots, s_m\}$, is the desired generating set for $S_{\lambda}$, which we now show. Note that $u_i \in H \cap \cO_\lambda$ by taking $a_i = 1$ and $a_j = 0$ for $j \neq i$. Therefore, the set $\{s_1, \ldots, s_m\}$ contains the set $\{u_1, \ldots, u_k\}$.
    
    Now, pick any $u = \sum_{i=1}^{k}b_{i}u_{i} \in S_{\lambda}$, where $b_i \in \mathbb{Q}_{\ge 0}$ and $u \ne 0$. Then each $b_{i}$ can be decomposed into an integer and non-integer part $b_{i} = n_{i} + r_{i}$, where $n_{i} \in \Z_{\ge 0}$ and $r_{i} \in [0,1)$, and
    \[
      u = \left(\sum_{i=1}^{k}n_{i}u_{i}\right) + \left(\sum_{i=1}^{k}r_{i}u_{i}\right).
    \]
 We aim to show that $u$ can be written as an non-negative integer sum of $\{s_1, \ldots, s_m\}$. Since $u_i \in \{s_1, \ldots, s_m\}$ for all $i$, the left part of the summation,  $\sum_{i=1}^{k}n_{i}u_{i}$, is a non-negative integer sum of $\{s_1, \ldots, s_m\}$. In addition, $\sum_{i=1}^{k}r_{i}u_{i} \in H$ by definition.
    However, $u- \sum_{i=1}^{k} n_{i}u_{i} \in \cO_{\lambda}$, so that $\sum_{i=1}^{k}r_{i}u_{i} \in H \cap \cO_{\lambda}$. This means that $\sum_{i=1}^{k}r_{i}u_{i}$ is one of the $s_j$ itself, which concludes the proof.  
    \end{proof}
    
    \begin{PROP}[Slim Cone Lemma; cf. {\cite[Section 4]{Thurston2014Entropy}}]
    \label{prop:SlimCone}
    Let $s \in S_\lambda$. Then there exists $N_s>0$ such that $n \ge N_s$ implies that $\lambda^{n} \cdot S_\lambda \subset s+S_\lambda$.
    \end{PROP}
    \begin{proof}
        Since $S_\lambda$ is finitely generated (\Cref{prop:Gordan}), we can let $S_\lambda = \<s_1,\ldots,s_m\>$. To show the conclusion, it suffices to prove the following claim:
        \begin{CLAIM} \label{claim:Slim_subclaim}
            Fix $s \in S_\lambda$. For each $i =1,\ldots,m$, there exists $N_i>0$ such that $n \ge N_i$ implies that $\lambda^n \cdot s_i \in s + S_\lambda$.
        \end{CLAIM}
        \Cref{claim:Slim_subclaim} then implies the conclusion by taking  $N_s:=\max\{N_1,\ldots,N_m\}$. Then for any $t = a_1s_1+\ldots+a_ms_m \in S_\lambda$ with $a_i \in \Z_{\ge 0}$ for all $i$, it follows that for $n > N_s$,
        \[
            \lambda^n \cdot t = \lambda^n \cdot \left(\sum_{i=1}^m a_is_i\right)
            =\sum_{i=1}^m a_i (\lambda^n \cdot s_i) \in s + S_\lambda
        \]
        because $\lambda^n \cdot s_i \in s+S_\lambda$ for each $i$ and $s+S_\lambda$ is a semigroup.
        
         To prove \Cref{claim:Slim_subclaim}, we need another claim: 
         
        \begin{CLAIM} \label{claim:Slim_attracting}
            For any $s \in S_\lambda$, the projection of
            $s+KR_\lambda$ surjects onto the interior of $\P(KR_\lambda)$.
        \end{CLAIM} 
        
        \begin{proof}[Proof of \Cref{claim:Slim_attracting}]
        \renewcommand{\qedsymbol}{$\triangle$}
        Write $s=b_1s_1+\ldots+b_ms_m$ for some nonnegative integers $b_1,\ldots,b_m$. 
        Any interior point of $\P(KR_\lambda)$ can be represented by an interior point of $KR_\lambda$.
        Let $t$ be an interior point of $KR_\lambda$. We will find $t' \in s+KR_\lambda$ such that $\bar{t}=\bar{t'}$. Write $t = a_1s_1 + \ldots + a_ms_m$. Then $a_1,\ldots,a_m >0$, because $t$ does not lie on any face of $KR_\lambda$. Therefore, there exists a large $R$ such that $Ra_i > b_i$ for every $i=1,\ldots,m$.
        Then we claim that
        \[
            Rt = Ra_1s_1 + \ldots + Ra_ms_m \in s+KR_\lambda.
        \]
        Indeed, $Rt-s = \sum_{i=1}^m(Ra_i-b_i)s_i$ has nonnegative coefficients, so $Rt-s \in KR_\lambda$. Therefore, $\bar{Rt} \in \P(s+KR_\lambda)$ and $\bar{Rt}=\bar{t} \in \textrm{int}(\P(KR_\lambda))$, as desired.
        \end{proof}
        
        Now we prove \Cref{claim:Slim_subclaim}. Note the leading eigenvector $\bar{v_1}$
        is the unique attracting fixed point in $\P(KR_\lambda)$ under the $\lambda$-multiplication action. \Cref{claim:Slim_attracting} implies that $\bar{v_1} \in \P(s + KR_\lambda)$. Then $\{\bar{\lambda^n \cdot s_i}\}_{n \in \Z^+}$ converges to $\bar{v_1}$. Hence, eventually for some large $N_s \in \Z^+$, we can say whenever $n\ge N_s$, it follows that $\bar{\lambda^n \cdot s_i}$ is arbitrarily close to $\bar{v_1}$ in $\P(s+KR_\lambda)$. Together with the fact that $\lambda^n \cdot s_i \in S_\lambda$ for all $n>0$, we obtain
        \[
        n \ge N_s \qquad \Longrightarrow \qquad \lambda^n \cdot s_i \in s+S_\lambda,
        \]
        proving \Cref{claim:Slim_subclaim} and concluding the proof.
    \end{proof}
    
    \subsection{Uniformly $\lambda$-expanding Star Map -- The proof of \Cref{thm:astmapPerron}}\label{ssec:AST_Uniform_Expanding}
        
In this section, we construct a uniformly $\lambda$-expanding star map for any Perron number $\lambda$ to prove \Cref{thm:astmapPerron}, which we restate here:
\setcounter{THME}{3}
\begin{THME}[Uniformly Expanding Star Maps for Perrons]\label{thm:perronexpander}
    Let $\lambda$ be a Perron number. Then there exists $n > 0$ and a star map $f: \ast_n \to \ast_n$ such that $f$ is a uniform $\lambda$-expander with mixing incidence matrix.
\end{THME}

\begin{proof}
Use \Cref{lem:even} to find $N,n_0 \in \Z^+$ with $N>n_0$ such that $\lambda^{N}\ \equiv \lambda^{n_0} \mod 2\cO_\lambda$.
Recall by \Cref{prop:Gordan}, we can write $S_\lambda = \<s_1,\ldots,s_m\>$. Set $T= s_1+\ldots+s_m$, and for each $k =1,\ldots,m$ define:
\[
    g_k = \lambda^{n_0} \cdot s_k + 2(T+\lambda \cdot s_k) \in S_\lambda. 
\] 
For each $k=1,\ldots,m$, \Cref{prop:SlimCone} with $s=g_k$ yields $N_k \in \Z^+$ such that $n \ge N_k$ implies that $\lambda^n \cdot s_k \in g_k + S_\lambda$. 

Now, letting $M:=p(N-n_0)+N$, where $p$ is sufficiently large integer so that $M \ge \max\{N_1,\ldots,N_m\}$, we have that
$ \lambda^M \cdot s_k \ \in \ g_k + S_\lambda,$
for all $k=1,\ldots,m$. We claim that we also have that $\lambda^{M} \equiv \lambda^{n_0} \mod 2\cO_\lambda$ since
\begin{align*} 
     \lambda^{p(N-n_0) + N} &\equiv \lambda^{p(N - n_0) + n_0} \\
     &\equiv \lambda^{pN-(p-1)n_0}  \\
     &\equiv \lambda^{(p - 1)(N - n_0)+N}\\
     &\phantom{{}\equiv{}}\vdots\\
        &\equiv \lambda^{(p - 2)(N - n_0)+N}\\
         &\phantom{{}\equiv{}}\vdots\\
     &\equiv \lambda^{N} \equiv \lambda^{n_0} \mod 2\cO_\lambda.
\end{align*}

Hence, for each $k = 1,\ldots,m$,
\[
    \lambda^M \cdot s_k - \lambda^{n_0} \cdot s_k \in 2(T+\lambda \cdot s_k) + S_\lambda,
\]
where the left hand side is $(\lambda^M - \lambda^{n_0})\cdot s_k \equiv 0 \mod 2\cO_\lambda$. This implies that $(\lambda^M - \lambda^{n_0})s_k$ is written as a linear combination of the $\{s_i\}_{i=1}^m$ with \emph{even} coefficients. However, as $2(T+\lambda \cdot s_k)$ already has even coefficients so that for each $k =1,\ldots,m$, there must exist $e_1^{(k)},\ldots,e_m^{(k)} \in \Z_{\ge 0}$ such that
\begin{align*}
    (\lambda^M - \lambda^{n_0})\cdot s_k &= 2(T+\lambda\cdot s_k) + \sum_{i=1}^m 2e_i^{(k)}s_i \\
                                       &= 2\lambda \cdot s_k + \sum_{i=1}^m (2e_i^{(k)}+2)s_i.
\end{align*}
Thus, we obtain the following key identity that will be used to construct the desired uniformly $\lambda$-expanding star map:
\begin{equation} \label{eq:recipe} \tag{$\dagger$}
    \lambda^M \cdot s_k = \left[\sum_{i=1}^m (2e_i^{(k)}+2)s_i\right] + (2\lambda \cdot s_k) + (\lambda^{n_0} \cdot s_k) , \qquad \text{ for } k=1,\ldots,m.
\end{equation}

Now we begin the construction of the star map using \Cref{eq:recipe}. Consider a star graph $\ast_{mM}$ with $mM$ tips. For $k=1,\ldots, m$ and $i = 1,\ldots,M$, label each edge by $(s_k,i)$ and set its length to be: 
\[
\|(s_k,i)\| = \lambda^{i-1}\cdot s_k \in S_\lambda \subset \R^+,
\]
where the last containment comes from the usual embedding $S_\lambda = (KR_\lambda \cap \cO_\lambda)\setminus\{\mathbf{0}\} \hookrightarrow (KR_\lambda \cap \Q(\lambda))\setminus \{\mathbf{0}\} \hookrightarrow \R^+$.

Now define the star map $f_{\lambda}:\ast_{mM} \to \ast_{mM}$ by mapping each edge $(s_k,i)$ to $(s_k,i+1)$ when $i<M$ and mapping $(s_k,M)$ to an edge path as follows:
\[
    (s_k,M) \quad\overset{f_\lambda}{\longmapsto}\quad
    (2e^{(k)}_{k}+2)(s_{k},1)\sqcup \bigsqcup_{i\neq k}(2e_i^{(k)}+2)(s_i,1) \sqcup 2(s_k,2) \sqcup (s_{k},n_{0}+1)
\]

More precisely, $f_\lambda$ sends each edge $\{(s_k,i)\}_{i=1}^{M-1}$ to the next edge corresponding to $s_{k}$, and sends the $M$-th edge $(s_k,M)$ to an edge path traversed in the following order:

\begin{enumerate}[(i)]
   \item First, it maps $2e_k^{(k)}+2$ times over the edge $(s_k,1)$,
    \item next, for each $j \neq k$ it goes $2e_j^{(k)}+2$ times over the edge $(s_j,1)$,
    \item then, it maps 2 times over the edge $(s_k,2)$,
    \item and finally, it maps over the edge $(s_{k},n_0+1)$ once.
\end{enumerate}

The choices of lengths of edges in $\ast_{nM}$ guarantee that $f_{\lambda}$ is uniformly $\lambda$-expanding. We immediately have
$\|f_\lambda ((s_k,i))\| = \|(s_k,i+1)\| = \lambda^i s_k = \lambda \|(s_{k},i)\|$ for $1 \le i < M$. For $i = M$, by \Cref{eq:recipe} we have
\[
\|f_\lambda((s_k,M))\| =\left[\sum_{i=1}^m (2e_i^{(k)}+2)s_i\right] + (2\lambda \cdot s_k) + (\lambda^{n_0} \cdot s_k)= \lambda^M s_k = \lambda\|(s_{k},M)\|.
\]

Now we verify that $f_\lambda$ is a star map with mixing incidence matrix. First, the fact that $f_\lambda$ traverses all edges but one an even number of times implies that $f_\lambda$ fixes the center vertex (or equivalently, preserves the bipartite structure).
By construction it follows that every edge is the first element of the image edge path of some edge (i.e., the first edge map $f_1$ is a permutation).
 
Finally, to see that $f_\lambda$ has mixing incidence matrix, we observe that for each edge $(s_k,i)$ (regarding $i$ as an integer modulo $M$):
\begin{align*}
    f_\lambda^M((s_k,i)) &\supset \bigsqcup_{k=1}^m (s_k,i) \\
    f_\lambda^{2M}((s_k,i)) &\supset \bigsqcup_{k=1}^m (s_k,i) \cup \bigsqcup_{k=1}^m (s_k,i+1) \\
    f_\lambda^{3M}((s_k,i)) &\supset \bigsqcup_{k=1}^m (s_k,i) \cup \bigsqcup_{k=1}^m (s_k,i+1) \cup \bigsqcup_{k=1}^m (s_k,i+2)\\
    &\vdots \\
    f_\lambda^{M^2}((s_k,i)) &\supset \bigsqcup_{i=1}^M\bigsqcup_{k=1}^m (s_k,i) = \ast_{mM},\\
\end{align*}
which implies that $f_\lambda$ has mixing incidence matrix. Indeed, denoting by $A_\lambda$ the incidence matrix of $f_\lambda$, its power $A_\lambda^{M^2}$ is a positive matrix by the above observation. This concludes the proof.
\end{proof}

The fact that the incidence matrix for $f_\lambda$ is mixing will be used to prove \Cref{LEM:PerronMixing}, which states that the incidence matrix of its \emph{split} $S(f_\lambda)$ is as well.

Next we expand \Cref{thm:astmapPerron} to \emph{weak Perron numbers} via the following lemma.
    
\begin{LEM}[Uniformly Expanding Star Maps for Weak Perrons]
\label{lem:WeakPerron_Expanding_Ast}
    Let $\lambda$ be a weak Perron number. Then there exists $k > 0$ and a star map $f: \ast_k \to \ast_k$ such that $f$ is a uniform $\lambda$-expander.
\end{LEM}

\begin{proof}
    If $\lambda$ is a weak Perron number, then by \Cref{prop:PnWP}, there exists some $N$ such that $\lambda^N$ is a Perron number. Then by \Cref{thm:astmapPerron} we can construct a star map $f:\ast_n \to \ast_n$ for some $n >0$, which is $\lambda^N$-uniformly expanding.
    
    Now take $N$-copies $C^0,\ldots,C^{N-1}$ of $\ast_n$, and glue them along the center vertices to form a star $\ast_k=\ast_{Nn}$ with $Nn$ tips.
    
    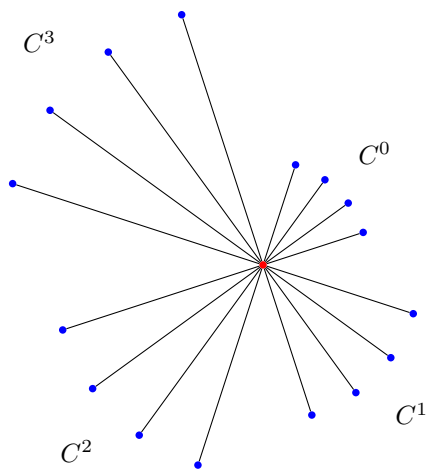
\begin{figure}[htbp]
        \centering
        \begin{tikzpicture}[scale=.7]
            \node[circle,fill=red,scale=0.3] at (360:0mm) (center) {};
            \foreach \n in {1,...,20}{
                \pgfmathparse{Mod(\n,5)==0?1:0}
                \ifnum\pgfmathresult=0
                    \pgfmathtruncatemacro\k{((\n/5)+2)}
                    \node[circle,fill=blue,scale=0.3] at ({-(\n)*360/20+90}:\k cm)     (n\n) {};
                    \draw (center)--(n\n);
                \else
                    \pgfmathtruncatemacro\m{\n/5-1}
                    \pgfmathtruncatemacro\l{\m+3}
                    \node at ({(7.5-\n)*360/20}:\l cm) {$C^{\m}$}; 
                \fi
            }
        \end{tikzpicture}
        \caption{Gluing four copies of $\ast_4$ to form $\ast_{16}$. The length of each edge in $C^{i+1}$ is given by multiplying by $\lambda$ to the corresponding edge in $C^{i}$.}
        \label{fig:glued_asterisks}
    \end{figure}

    Decide the lengths on each edge in $C^0$ according to the construction of $\ast_n$ from \Cref{thm:astmapPerron}.
    Then for $i=1,\ldots,N-1$, set the edge length of edges in $C^i$ to be exactly $\lambda^{i}$ times the length of the corresponding edge of $C^0$. Now define $f_N:\ast_{k}\to \ast_{k}$ as follows. For $0 \le i \le N-2$, send each edge of $C^i$ to the corresponding edge in the next copy, $C^{i+1}$, i.e. $f_N$ simply shifts $C^i$ to $C^{i+1}$. Then by construction, $f_N$ is uniformly $\lambda$-expanding on $C^0,\ldots,C^{N-2}$.
    Next, to make $f_N$ uniformly $\lambda$-expanding on $C^{N-1}$ as well, define $f_N$ by mapping each edge of $C^{N-1}$ to an edge path of $C^0$ following the recipe given by the map $f:\ast_N \to \ast_N$ in \Cref{thm:astmapPerron}, which is $\lambda^N$-uniformly expanding. In particular, we apply $f$ to the edge of $C^{N-1}$, uniformly expanding the length of the edge by $\lambda^N$, and then shift the image of edge path to corresponding edge path in $C_0$, which shrinks the length of the edge path by $\lambda^{N-1}$. Therefore, each edge $e$ in $C^{N-1}$ will be mapped to an edge path in $C^0$ whose length is $(\lambda^N)/\lambda^{N-1}=\lambda$ times the length of the edge $e$ in $C^{N-1}$.
  Given a weak Perron number $\lambda$, this concludes the construction of  $f_N:\ast_{k} \to \ast_{k}$ which is uniformly $\lambda$-expanding.
\end{proof}

\section{Prototype Graph}\label{sec:Prototypes}\vspace{1em}
    \subsection{Constructing the prototype maps}
    \label{ssec:PRO_Definition}
Let $P_7$ denote the bipartite graph with two 
vertices $v_0, v_1$ and seven edges as in \Cref{fig:thurstonttmap}. We will refer to this graph as the \textbf{prototype graph}. Orient the edges
according to the bipartite structure, \eg with initial
vertex $v_0$ and terminal vertex $v_1$; let 
$a, b, \ldots, g$ label the edges with respect to this 
orientation, and  $A, B,\ldots, G$ 
 denote the opposite orientation.  
We endow $P_7$ with a traintrack structure as 
shown in \Cref{fig:thurstonttmap}. Note that all of the turns between the $a,b,$ and $c$ edges are legal at both of the vertices.

\begin{figure}[ht!]
	      \centering
  \makebox[\textwidth][c]{\scalebox{1}{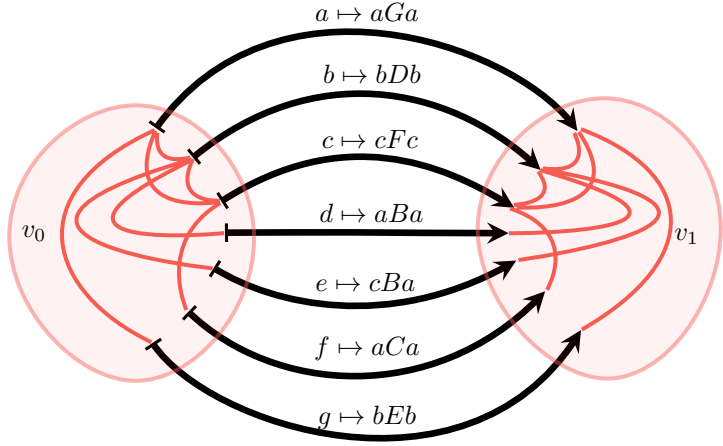}}
		    \caption{The traintrack structure on $P_{7}$ from Thurston's original paper \cite{Thurston2014Entropy}.  The red paths between edges indicate the legal turns. }
	    \label{fig:thurstonttmap}
\end{figure}

Let $\phi_1$ be the identity map on $P_7$. Now we define a set of traintrack maps $\{\phi_{3+2m}\}_{m=0}^\infty$ on the prototype graph $P_7$.
For $m \geq 0$ let 
$\phi_{3 + 2m} : P_7 \to P_7$ be the
taut graph map defined as follows:

\begin{equation*}
\phi_{3 + 2m} : \begin{cases}
a &\longmapsto \quad aG(aB)^ma, \\
b &\longmapsto \quad bD(bC)^mb, \\
c &\longmapsto \quad cF(cA)^mc, \\
d &\longmapsto \quad aB(aB)^ma, \\
e &\longmapsto \quad cB(aB)^ma, \\
f &\longmapsto \quad aC(aB)^ma, \\
g &\longmapsto \quad bE(bA)^mb.
\end{cases}
\end{equation*}

Since edges are mapped to paths of odd length, each $\phi_{3+2m}$ preserves the bipartite structure of 
$P_7$, and in fact fixes the vertices $v_0,v_1$ pointwise. We will call these maps $\{\phi_{n}\}$, for $n$ odd and positive, the \textbf{prototype maps}.

We claim the prototype maps are all traintrack maps. 
Since $\phi_{1}$ is the identity map it is a traintrack map. One can then check that for $\phi_{3+2m}$ every legal turn is sent to a legal turn. Note that in determining the action of a map on turns we only need to consider the first and last edges of the image edge-paths. This allows us to check that $\phi_{3+2m}$ is a traintrack map for all $m \ge 0$.  For example, we check $\phi_{3+2m}$ maps the legal turns at $v_{0}$ that involve the $a$ edge to legal turns at $v_0$.
\begin{align*}
    Ba &\mapsto Ba \\
    Ca &\mapsto Ca \\
    Ga &\mapsto Ba.
\end{align*}
Now one can check using \Cref{fig:thurstonttmap} that each edge is mapped to a legal path. This check is possible for all $m \ge 0$ since at most only \emph{four} distinct turns occur in the image of each edge. For example, $\phi_3(a)$ consists of two kinds of turns $aG,Ga$ and $\phi_{3+2m}(a)$ with $m\ge 1$ consists of four kinds of turns: $aG,Ga,aB$ and $Ba$, which are all legal as shown in \Cref{fig:thurstonttmap} and this is independent of the value of $m \ge 0$. Lastly, we note that the prototype maps are indeed taut since they are local embeddings on the interiors of edges. Again, this is independent of $m$ and is checked directly by noting that the image of each edge is \emph{reduced} as a word in the free group $\pi_1(P_7)$.

\subsection{Homotopy equivalence via folding.}

In this section we prove that the prototype maps are in fact homotopy equivalences, and therefore, induce automorphisms of $\pi_1(P_7) \cong \F_6$, the free group of rank 6.

\begin{LEM} \label{lem:prototypehe}
    All of the prototype maps $\{\phi_{n}\}$, for $n$ positive and odd, are homotopy equivalences.
\end{LEM}

\begin{proof}
    The map $\phi_{1}$ is the identity map and hence a homotopy equivalence. We will check directly using Stalling folds that $\phi_{3+2m}$ is a homotopy equivalence for $m\geq 0$.  Recall that if we perform only Type I folds while decomposing a graph map \'{a} la Stallings and the resulting immersion, $\psi$, is a homotopy equivalence,
    then the original map is a homotopy equivalence. This argument is largely a ``proof by picture" and so we direct the reader to \Cref{fig:compiled_phi3} throughout the proof. We first show that $\phi_{3}$ is a homotopy equivalence. Begin by subdividing the edges of the domain graph via the full pre-image of the vertices. Then $\phi_{3}$ is described by \Cref{fig:compiled_phi3}. Note the edges $d,e,f,g$ are in black as they will never be folded.
    
    \begin{figure}[ht!]
    \centering
    \makebox[\textwidth][c]{
    \includegraphics[width=\textwidth]{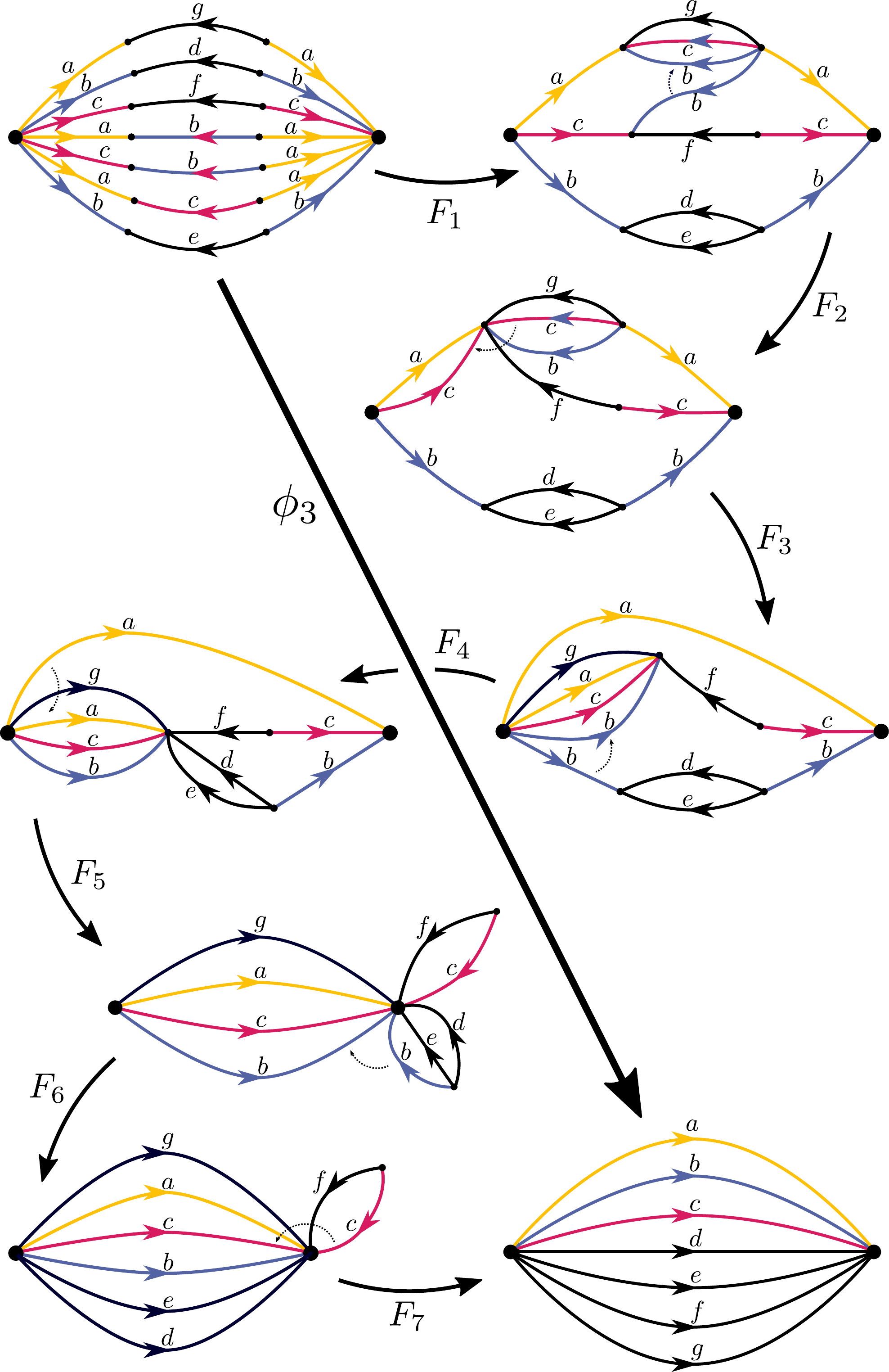}}
    \caption{Folding $\phi_3$. All folds $F_1,\ldots,F_7$ are of Type I.}
    \label{fig:compiled_phi3}
\end{figure} 
    
   We begin with all of the possible folds on edges adjacent to the left- and rightmost vertices. That is to say, we fold all edges with the same label and orientation at a vertex. The first four collection of folds, labeled $F_{1}, F_{2}, F_{3}$ and $F_{4}$, are performed in \Cref{fig:compiled_phi3}. For readability we will omit the $a,b,c$ labels on the yellow, green, and pink edges in the figures from now on. We point out that for these first four folding maps we are only ever folding edges that are distance at most two from the leftmost vertex or one from the rightmost vertex.
    This is unimportant right now, but will be important in the proof of \Cref{prop:splithe}, when we verify that our \emph{split maps} are homotopy equivalences.

    \Cref{fig:compiled_phi3} shows the next three folds, $F_5, F_6,$ and $F_7$ that are performed. The map $\psi$ from \Cref{thm:stallings} after these seven folding maps is the identity up to permuting edges. Thus, $\phi_{3}$ is a homotopy equivalence.
    
    Next we turn our attention to an arbitrary $\phi_{3+2m}$ with $m \ge 1$. \Cref{fig:compiled_phi_general} shows the map $\phi_{3+2m}$ graphically.
    
    \begin{figure}[ht!]
    \centering
    \makebox[\textwidth][c]{
    \includegraphics[width=1.2\textwidth]{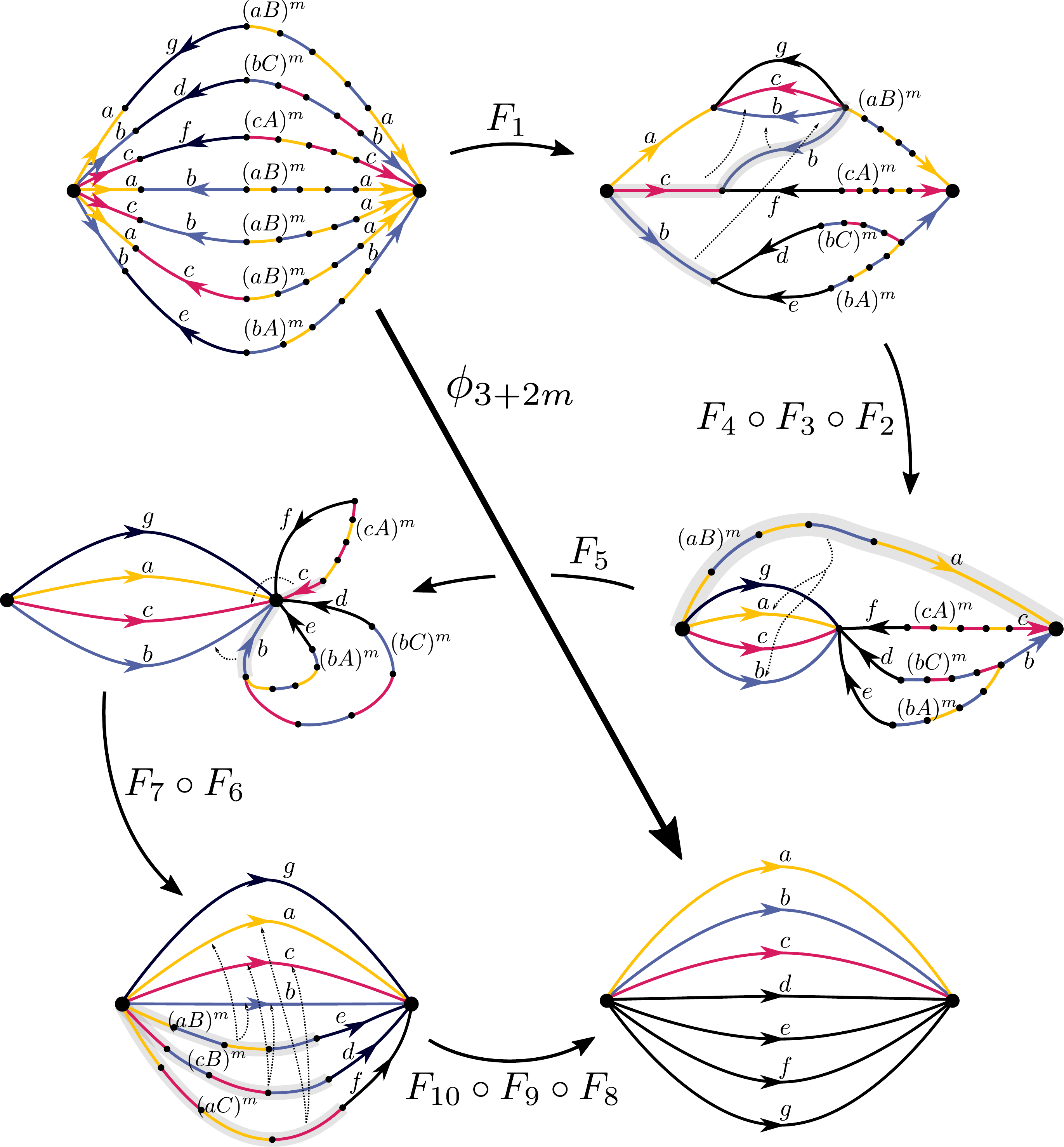}
    }
    \caption{Folding $\phi_{3+2m}$ for $m \ge 1$. All folds $F_1,\ldots,F_{10}$ are of Type I. The highlighted grey edge paths at each stage represent edges involved in the folding maps being described.} 
    \label{fig:compiled_phi_general}
    \end{figure}
    
     We begin with the first four folding maps that are effectively the same folds as in the $\phi_{3}$ case. The main difference is that for $F_{1}$ we fold not just the first edge adjacent to the rightmost vertex, but instead fold the entire edge path of length $2m+1$ adjacent to this vertex. For example, $\phi_{3+2m}(a)$ and $\phi_{3+2m}(d)$ both end in $(aB)^ma$ so the portions of the edge paths corresponding to these $2m+1$ characters are folded. 
     The next three maps, $F_{2},F_{3},F_{4}$, are exactly the same as in the $\phi_{3}$ case. (Compare $F_2,F_3,F_4$ in \Cref{fig:compiled_phi3} with those in \Cref{fig:compiled_phi_general}.)

    Next we perform a new type of folding map called a \emph{wrapping} map. See \Cref{fig:compiled_phi_general}. After performing $F_4\circ F_3\circ F_2$, there is a loop labelled by $aB$ beginning at the left-most vertex. The map $F_{5}$ will consist of $2m+1$ Type I folds that take an edge path labelled by $(aB)^ma$ and first ``wrap" it around this loop labeled by $aB$ a total of $m$-times before folding the last $a$-edge. These folds are possible because the $a$-edge at the start of the edge path labelled by $(aB)^m$ and the $a$-edge of the loop $aB$ we are wrapping the edge path around begin at the same vertex.
    Next, $F_{6}$ and $F_7$ are simple single folds along a $b$-edge and a $c$-edge respectively. Then $F_{8},F_{9},$ and $F_{10}$ are again ``wrapping" maps. Namely, the map $F_8$ folds $m$ times around $cB$, the map $F_9$ folds $m$ times around $aB$, and the map $F_{10}$ folds $m$ times around $aC$. The single fold performed by the $F_6$ map eliminates the need to do a single fold in $F_8$ and $F_9$. Note that it is essential to perform the single fold in the $F_6$ map before $F_8$ in order to ensure that the edge path labelled by $(cB)^m$ and the $c$-edge of the loop $cB$ begin at the same vertex so that folding is possible. The same applies to the edge path labelled by $(aB)^m$ that is wrapped around $aB$ in $F_9$. Similarly, $F_{10}$, wrapping $(aC)^m$ over $aC$ requires the single-edge fold $F_7$.
    In this way we see that we obtain the original graph and that $\psi$ from \Cref{thm:stallings} is the identity map up to permuting edges after performing all of these Type I folds. We conclude that $\phi_{3+2m}$ is also a homotopy equivalence for all $m \ge 1$. 
    \end{proof}

\section{Splitting graphs}\label{sec:Splitting}\vspace{1em}
Given a weak Perron number $\lambda$, we constructed a graph map $f_\lambda$ on a star graph that is uniformly $\lambda$-expanding and a homotopy equivalence in \Cref{sec:Asterisk_Maps}. Though it is close to being the desired map for \Cref{thm:thurston_main}, it is far from being a traintrack map because of the large amounts of backtracking in the star maps.
In this section, we edit $f_\lambda$ by ``blowing up" each edge of the star graph to the prototype graph $P_7$ and use the prototype maps of \Cref{sec:Prototypes} to resolve the backtracking and obtain a traintrack map (\Cref{prop:splitttmap}). During this process, we carefully preserve certain properties of $f_\lambda$, specifically uniformly $\lambda$-expanding (\Cref{prop:splitexpanding}) and being a homotopy equivalence (\Cref{prop:splithe}). This process of blowing up a graph is named \emph{splitting} by Thurston in \cite[Section 9]{Thurston2014Entropy} and we will discuss the process in \Cref{ssec:SPL_Graphs}.

\subsection{Split Graphs and Split Maps}\label{ssec:SPL_Graphs}

\begin{DEF}
Given a bipartite metric graph $\Gamma$ we define the \textbf{split graph}, $S(\Gamma)$, by replacing each edge of $\Gamma$ with a copy of the prototype graph $P_7$ from \Cref{sec:Prototypes} whose edges have the same length as the replaced edge. That is, if the edges of $\Gamma$ are enumerated $x_{1},x_{2},\ldots$, replace the edge $x_{i}$ between vertices $v_{0}$ and $v_{1}$ with $7$ edges labelled $a_i,b_i,\ldots, g_i$ each of which has the same length as $x_i$.
See \Cref{fig:splitex1} for an example of a split graph.
\end{DEF}

\begin{figure}[ht]
	    \centering
	    \def\svgwidth{\columnwidth}
		    \scalebox{.75}{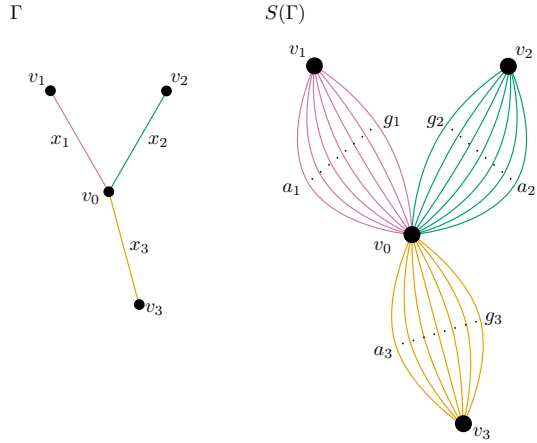}
		    \caption{Example of a split star graph. The initial graph, $\Gamma$, is on the left and the split graph, $S(\Gamma)$, is on the right.}
	    \label{fig:splitex1}
\end{figure}

$S(\Gamma)$ inherits a traintrack structure from the traintrack structure on the prototype graph in the following way: turns are legal in $S(\Gamma)$ if and only if the turns without subscripts define a legal turn in the prototype. For example, a turn of the form $a_iB_j$ is legal as the turn $aB$ is legal in the prototype graph $P_7$, and a turn of the form $a_iF_j$ is not, as the turn $aF$ is illegal in $P_7$.

\begin{DEF}
Given a self-map $f:\Gamma\longrightarrow\Gamma$ which preserves the bipartite structure and does not collapse any edges, we define the \textbf{split map}
    \[
    S(f):S(\Gamma)\longrightarrow S(\Gamma)
    \]
as follows.  Let $y_i$ be an edge of $S(\Gamma)$ corresponding to the edge $x_i$ of $\Gamma$, where $y \in \{a, b, \ldots, g\}$. For an edge path $w$ in $\Gamma$, denote by $\|w\|$ the length of $w$ when each edge of $\G$ is endowed with length 1. Now, if $\|f(x_i)\| = \ell$, then $S(f)$ sends $y_i$ to the edge path given by $\phi_\ell(y)$ with subscripts identical to the subscripts of the edge path $f(x_{i})$. Note that the word length of the edge path $f(x_i)$ determines which prototype map is used to define the image of $y_i$ for all $y \in \{a,b,\ldots,g\}$. Symbolically:
\[
    S(f)(y_i) := [\phi_{\|f(x_i)\|}(y)]_{f(i)},
\]
where for a path $\mathcal{P}$ in the prototype graph $P_7$, the notation $[\mathcal{P}]_{f(i)} \subset S(\Gamma)$ means that the path $\mathcal{P}$ is given the subscripts of the edge path $f(x_i)$. 

 Note, since $f$ preserves the bipartite structure, the word length of edge paths of the form $f(x_i)$ is always \emph{odd}, which makes our definition using prototype maps $\{\phi_{n}\}$ well-defined. Moreover, $S(f)$ preserves the induced bipartite structure on $S(\G)$.
\end{DEF}

\begin{EX}
Let $\Gamma$ be the star graph $\ast_{3}$.
Define a star map $f:\Gamma\longrightarrow\Gamma$ by:
\begin{align*}
    f(x_1)&=x_2X_2x_3,\\
    f(x_2)&=x_3,\\
    f(x_3)&=x_1X_1x_3X_3x_2.
\end{align*}

Then for the split map $S(f):S(\Gamma)\longrightarrow S(\Gamma)$, we need to use three different prototype maps for edges with three different subscripts. Namely, for $y \in \{a,b,\ldots,g\}$, we need to use $\phi_3$ for $S(f)(y_1)$ as $\|f(x_1)\|=3$, use $\phi_1=\id$ for $S(f)(y_2)$ as $\|f(x_2)\|=1$, and use $\phi_5$ for $S(f)(y_3)$ as $\|f(x_3)\|=5$. 

More precisely,
    \begin{align*}
    & S(f)(a_1) = a_2G_2a_3, && S(f)(b_1) = b_2D_2b_3, && S(f)(c_1)=c_2F_2c_3, && \\
    & S(f)(a_2) = a_3, && S(f)(b_2) = b_3, && S(f)(c_2)=c_3, && \\
    & S(f)(a_3) = a_1G_1a_3B_3a_2, && S(f)(b_3) = b_1D_1b_3C_3b_2, && S(f)(c_3)=c_1F_1c_3A_3c_2, && 
    \end{align*}
and so on for $y_i = d_i, e_i, f_i,$ and $g_i$, with $i = 1, 2, 3.$ 
\end{EX}

\begin{PROP}\label{prop:splitexpanding}
     If $f:\Gamma \to \Gamma$ is uniformly $\lambda$-expanding, then so is $S(f):S(\Gamma) \to S(\Gamma)$.
\end{PROP}

\begin{proof}
Essentially, this follows from the definition of the split map. Let $\ell$ be the length function on $\Gamma$. If $f(x_i)=x_{i_1}\cdots x_{i_k}$, then $S(f)(y_i)=z^{(1)}_{i_1}\cdots z^{(k)}_{i_k}$ for some $z^{(1)},\ldots,z^{(k)} \in \{a^{\pm 1},b^{\pm 1},\ldots,g^{\pm 1}\}$. By definition of $S(\Gamma)$, we have $\ell(x_{i_j})=\ell(z^{(j)}_{i_j})$ for all $j$, so
\[
    \ell\left(S(f)(y_i)\right)=\sum_{j=1}^k \ell(z^{(j)}_{i_j}) = \sum_{j=1}^k \ell(x_{i_j}) =\ell(f(x_i)) = \lambda \ell(x_i) =\lambda \ell(y_i),
\]
which proves that $S(f)$ is uniformly $\lambda$-expanding.
\end{proof}
    \subsection{Splitting Maps are Traintrack Maps}\label{ssec:SPL_Traintrack_Maps}
        It turns out that, with some mild conditions, $S(f)$ is always a traintrack map.

\begin{PROP} \label{prop:splitttmap}
    Let $f:\Gamma \rightarrow \Gamma$ be a map that preserves the bipartite structure and maps each edge to an edge-path of length at least $1$. Then $S(f)$ is a traintrack map.
\end{PROP}

\begin{proof}
    As before, we will denote by $\{x_i\}$ the edge set of $\Gamma$, and by $y_i$ an edge of $S(f)$ with $y \in \{a,\ldots,g\}$, corresponding to an edge $x_i$ of $\Gamma$. 
    
    Recall that the traintrack structure on $S(\Gamma)$ is defined so that a turn is legal if and only if the corresponding turn in the prototype graph $P_7$ obtained by forgetting subscripts is legal. 
    Hence, it automatically follows that every edge $y_i$ of $S(\Gamma)$ is sent to a legal path via $S(f)$, because forgetting the subscripts of the image $S(f)(y_i)$ is exactly $\phi_{\|f(x_i)\|}(y)$, which is a legal path since the prototype maps $\{\phi_n\}$ are traintrack maps. (See \Cref{ssec:PRO_Definition}.) 
    
    Therefore, to show $S(f)$ is a traintrack map it suffices to show that $S(f)$ sends a legal turn to a legal turn. First note that only turns of the form $y_iZ_j$, where $y \in \{ a, \ldots, g\}$ and $Z \in \{A, \ldots, G\}$ are legal in $S(\Gamma)$. Take such a legal turn. Then, again by definition, $yZ$ is a legal turn in $P_7$. Since all the legal turns in $P_7$ consist of at least one of the $a$-, $b$-, or $c$-edges, and $zY$ is legal if and only if $yZ$ is legal, we may assume without loss of generality that $y \in \{a,b,c\}$ (otherwise swap the roles of $y$ and $z$ in this proof). On the other hand, the turn $y_iZ_j$ will be mapped to the following turn, which is a concatenation of two edges:
    \[
        \left[\text{Last edge of }S(f)(y_i)\right]\cdot \left[\text{First edge of }S(f)(Z_j)\right].
    \]
    This turn is legal if and only if the corresponding turn in $P_7$
    \[
        \left[\text{Last edge of }\phi_{\|f(x_i)\|}(y)\right]\cdot \left[\text{First edge of }\phi_{\|f(x_j)\|}(Z)\right]
    \]
    (i.e., deleting the subscripts)
     is legal in $P_7$. The key observation is that by the definition of prototype maps $\{\phi_{3+2m}\}$ on the $a$, $b$, and $c$ edges of $P_7$, the last edge of $\phi_{3+2m}(y)$ is the same as $y$ regardless of the value $3+2m$. Therefore, we can replace the prototype map $\phi_{\| f(x_i) \|}$ for $y$ to $\phi_{\| f(x_j) \|}$, so the turn is actually identical to the following turn:
    \[
    \left[\text{Last edge of }\phi_{\|f(x_j)\|}(y)\right]\cdot \left[\text{First edge of }\phi_{\|f(x_j)\|}(Z)\right],
    \]
    which is legal because $yZ$ is a legal turn and $\phi_{\|f(x_j)\|}$ is a traintrack map. This concludes that $S(f)$ is a traintrack map.
    
    \end{proof}

\subsection{Splitting Star Maps}\label{ssec:SPL_Asterisk_Maps}
      Next we verify that $S(f)$ induces an automorphism of $\pi_{1}(S(\Gamma))$. That is, we check that $S(f)$ is a homotopy equivalence.
        
\begin{PROP} \label{prop:splithe}
    Let $*_{n}$ be a star graph and let $f:*_{n} \rightarrow *_{n}$ be a star map. Then $S(f)$ is a homotopy equivalence. 
\end{PROP}

\begin{proof}
    We use Stallings folds as in the proof for prototype maps (\Cref{lem:prototypehe}). First note the following two properties of star maps. A star map is always a permutation on the first edge map and each edge is mapped to an edge-path of odd length. In particular, the image of an edge under a star map is of the form $f(x_{r}) = x_{j_{1}}X_{j_{1}}\cdots x_{j_{m+1}}X_{j_{m+1}}x_{j_{m+2}}$, where $r$ and $j_1, \ldots, j_{m+2}$ are in $\{1, \ldots, n\}$, so that the edges in the image come in pairs (of an edge and its inverse) until the final edge.

    The split graph, $S(*_{n})$ is a ``flower" with $n$ ``petals", each of which is a copy of the prototype graph, $P_{7}$.
    We will first focus on a single petal corresponding to an edge $x_r$ of $\ast_n$ and begin folding there. The map on a petal is similar those given by \Cref{fig:compiled_phi_general} with the differences arising from the fact that the labels on the edges come equipped with subscripts. The subscripts are determined by the map $f$ on the original edges of the star graph as well. For instance, if $f(x_{r}) = x_{j_{1}}X_{j_{1}}\cdots x_{j_{m+1}}X_{j_{m+1}}x_{j_{m+2}}$, then
    \[
        S(f)(a_{r}) = a_{j_{1}}G_{j_{1}} a_{j_2}B_{j_2}\cdots a_{j_{m+1}}B_{j_{m+1}}a_{j_{m+2}},
    \]
    since $\phi_{3+2m}(a) = aG(aB)^ma$. This means that if $\|f(x_r)\| \ge 3$, then $S(f)(a_r)$ will traverse $a$-, $B$-, and $G$-edges in different petals. To initiate the folding process, we subdivide the edges according to their image under $S(f)$. On the $r$-th petal, we see the same sequence of subscripts on every single edge since the subscripts for the images $S(f)(a_r), S(f)(b_r), \ldots, S(f)(g_r)$ are all determined by $f(x_r)$. In case $\|f(x_r)\|=1$, say $f(x_r)=x_{r'}$, then $S(f)$ will just map the $r$-th petal to the $r'$-th petal. In this case, there is no need to subdivide the edges in the $r$-th petal, but we need to relabel the edges $y_r$ as $y_{r}'$ for $y \in \{a,\ldots,g\}$ before we fold.
    
    Recall the first four folding maps we performed in the proof of \Cref{lem:prototypehe} only fold edges that have the same subscript. To be precise, every fold that is performed through the maps $F_{1}, \ldots, F_{4}$ is between edges that lie in the two leftmost edges or in the rightmost edges in each petal. We refer again to \Cref{fig:compiled_phi_general} for a picture of these folds. Thus, after performing $F_1,\ldots,F_4$ on each petal we obtain a picture as in the middle of \Cref{fig:splitmapmidstep}. When $\|f(x_r)\|=1$, the maps $F_1,\ldots,F_4$ are just identity maps on the $r$-th petal.
    
    \begin{figure}[ht!]
	    \centering
		    \makebox[\textwidth][c]{\includegraphics[width=1.5\textwidth]{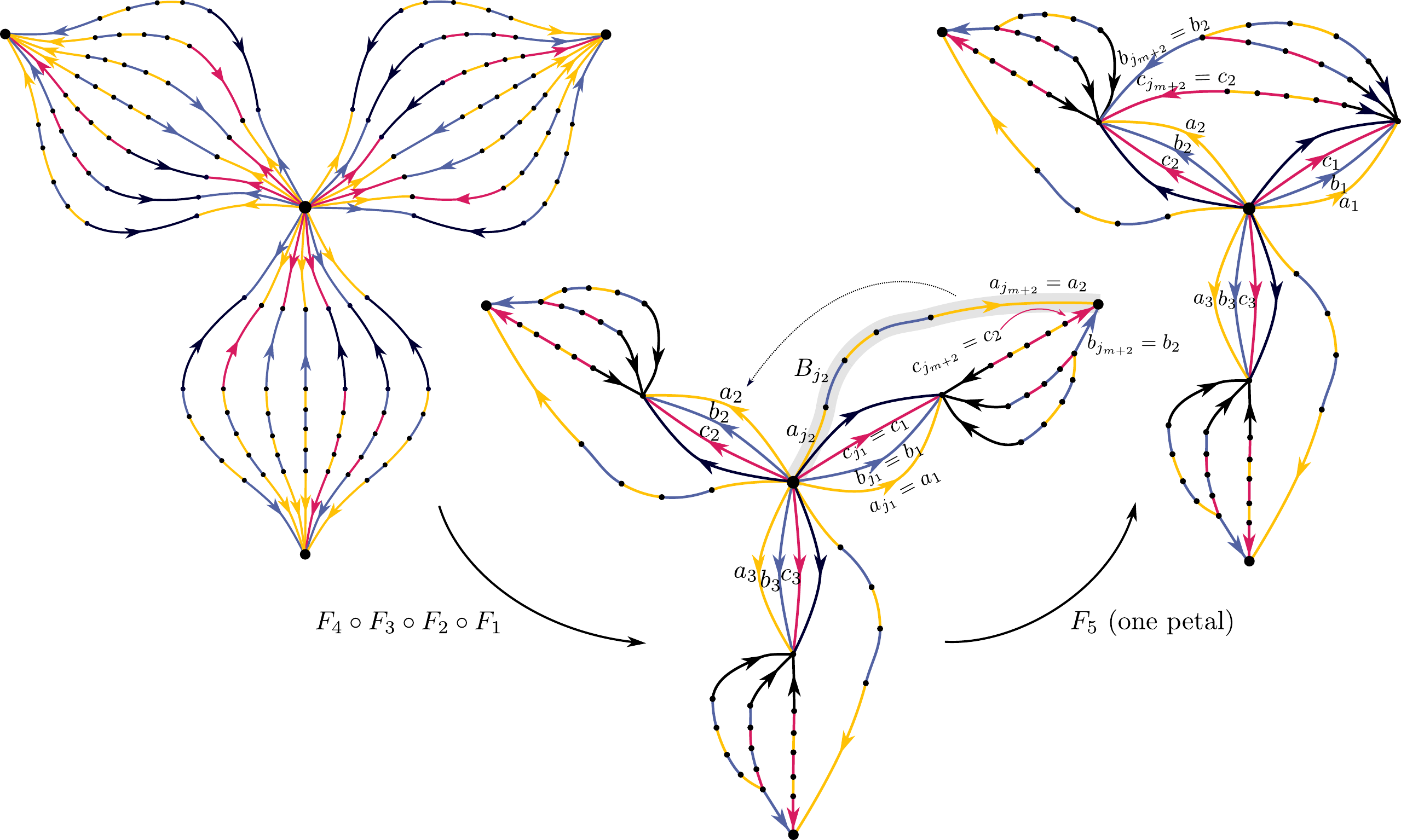}}
		    \caption{$(j_1=1,j_{m+2}=2)$. The graphs obtained after doing the four folds $F_1,\ldots,F_4$ on each petal, followed by performing $F_5$ on a single petal. Note that since the star map $f$ acts as a permutation on the first edges, all of the $a_{i}$, $b_{i}$, and $c_{i}$ appear and share the middle vertex as their initial vertex.}
	    \label{fig:splitmapmidstep}
    \end{figure}    
    
   The last few folding maps in the proof of \Cref{lem:prototypehe} were ``wrapping" maps that performed a sequence of folds wrapping edge paths of the form $(aB)^m,(cB)^m,$ and $(Ca)^m$ around loops $aB, cB,$ and $aC$, respectively. We could perform the first fold of $F_5$ because the first $a$-edge of $(aB)^m$ and the $a$-edge in the $aB$ loop that the path is wrapped around share an initial vertex. Then each subsequent fold of $b$- or $a$-edges has the same property. The subtlety in performing similar wrapping maps in \Cref{fig:splitmapmidstep} is that the edge paths now have subscripts and we must verify that the appropriate edges with the same subscript share an initial vertex so that we can perform each fold in the wrapping maps.
   
   The wrapping map $F_5$ in \Cref{lem:prototypehe} folds the edge path corresponding to $(aB)^ma$ in $\phi_{3+2m}(a)$ (all but the first two characters). Recall that 
   \[
   S(f)(a_{r}) = a_{j_{1}}G_{j_{1}} a_{j_2}B_{j_2}\cdots a_{j_{m+1}}B_{j_{m+1}}a_{j_{m+2}},
   \]
   and that $f$ is a permutation on the first edges, which implies that after performing $F_{4} \circ\cdots\circ F_{1}$ \emph{all} of the $a_{i}$, $b_{i}$, and $c_{i}$ edges are adjacent to the central vertex in \Cref{fig:splitmapmidstep} for $i= 1, 
   \ldots, n$.
   This means that we can perform the first two folds on the $a_{j_2}B_{j_2}$ edges of the folding map $F_5$, regardless of the value of $j_2$. After folding these two edges, the next pair $a_{j_3}B_{j_3}$ in the edge path now starts at the center vertex in \Cref{fig:splitmapmidstep} so that we can fold regardless of the value of $j_3$. Continuing in this way, we can perform $2m$ of the $2m+1$ single folds that constitute $F_5$. The final fold in $F_5$ in the proof of \Cref{lem:prototypehe} is along an $a$-edge. This final $a$-edge is labelled by $a_{j_{m+2}}$ and has the center vertex of the flower as its initial vertex after the first $2m$ folds of $F_5$ (see \Cref{fig:splitmapmidstep}). Thus, the final fold of $F_5$ can indeed be performed. The result of the folding map $F_5$ on the petal corresponding to $x_r$ is shown in \Cref{fig:splitmapmidstep} with the assumption that $j_1 = 1$ and $j_{m+2}=2$ for simplicity of the picture. The figure is far more complicated than that in the proof of \Cref{lem:prototypehe} since edges in the $r$-th petal are folded with edges in petals corresponding to other $x_k$'s.
   
   \begin{figure}[ht!]
       \centering
       \makebox[\textwidth][c]{\includegraphics[width=1.3\textwidth]{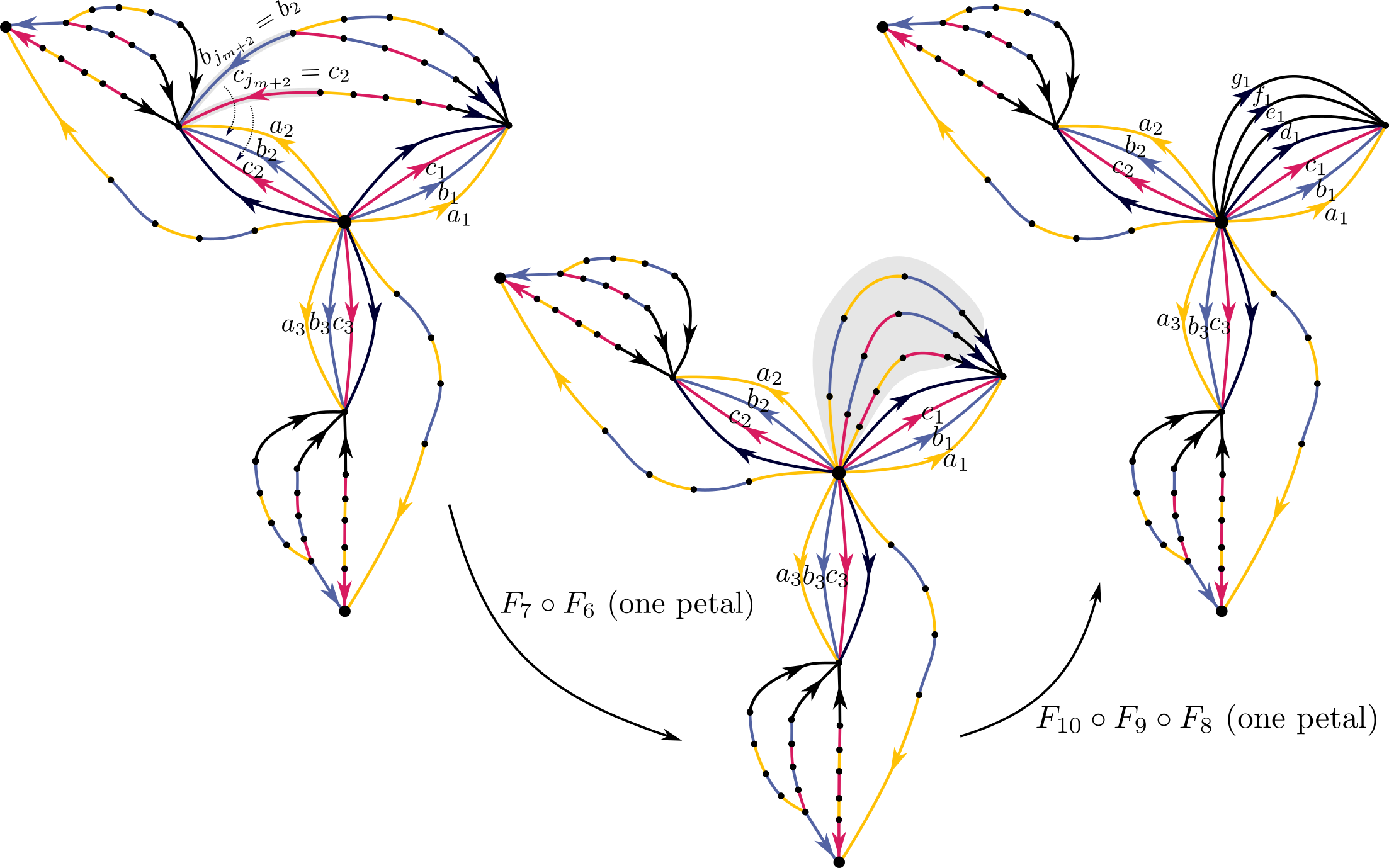}}
       \caption{$(j_1=1,j_{m+2}=2)$. Continuing from \Cref{fig:splitmapmidstep}, we perform $F_6$ and $F_7$ each to fold single edges $b_{j_{m+2}}$ and $c_{j_{m+2}}$ to obtain the middle graph. Then we do three ``wrapping-up'' folds $F_8,F_9$ and $F_{10}$ to obtain the rightmost graph, where the petal we chose to fold now looks a copy of $P_7$.}
       \label{fig:splitmapsecondstep}
   \end{figure}
   
   Moving forward, the $F_6$ map in the proof of \Cref{lem:prototypehe} is a fold involving the last $b$-edge in the edge path corresponding to $\phi_{3+2m}(b) = bD(bC)^{m}b$. The last $b$-edge in $S(f)(b_r)$ is labelled by $b_{j_{m+2}}$. In fact, due to the folds that were already performed, this edge has been identified with the last $b$-edge of $S(f)(g_r)$. Due to the final fold in $F_5$ described in the previous paragraph, this $b_{j_{m+2}}$ edge shares an initial vertex with another $b_{j_{m+2}}$ edge and so we perform the single fold and call it $F_6$ (see \Cref{fig:splitmapsecondstep}).
   Similarly, the final edge of the edge path corresponding to $\phi_{3+2m}(c) = cF(cA)^mc$ is labelled by $c_{j_{m+2}}$ and shares an initial vertex with another $c_{j_{m+2}}$ due to the folds of $F_5$. We perform this single fold, which we call $F_7$.
   Again, we refer the reader to \Cref{fig:splitmapsecondstep} where we assume $j_1 = 1$ and $j_{m+2}=2$ for simplicity.

   After $F_6$ and $F_7$ have been applied, the edge path $c_{j_{m+1}}B_{j_{m+1}}\cdots c_{j_2}B_{j_2}$ starts at the center vertex of the petal, and therefore, we can begin the folds that wrap around the appropriate loops of the form $c_iB_i$, just as we did in the proof of \Cref{lem:prototypehe}. Again, we are using the fact here that all of the $a_{i}$, $b_{i}$, and $c_{i}$ edges begin at the center vertex (see the argument for $F_5$ above). In the same way, we can perform the $F_9$ map once $F_6$ has been applied and wrap the edge path labelled by $a_{j_{m+1}}B_{j_{m+1}}\cdots a_{j_2}B_{j_2}$ around the appropriate $a_iB_i$ loops. Lastly, $F_{10}$ wraps the edge path labelled by $a_{j_{m+1}}C_{j_{m+1}} \cdots a_{j_2}C_{j_2}$ around the appropriate $a_iC_i$ loops, which is possible due to the single fold of $F_7$.

   \begin{figure}[ht!]
       \centering
       \makebox[\textwidth][c]{\includegraphics[width=\textwidth]{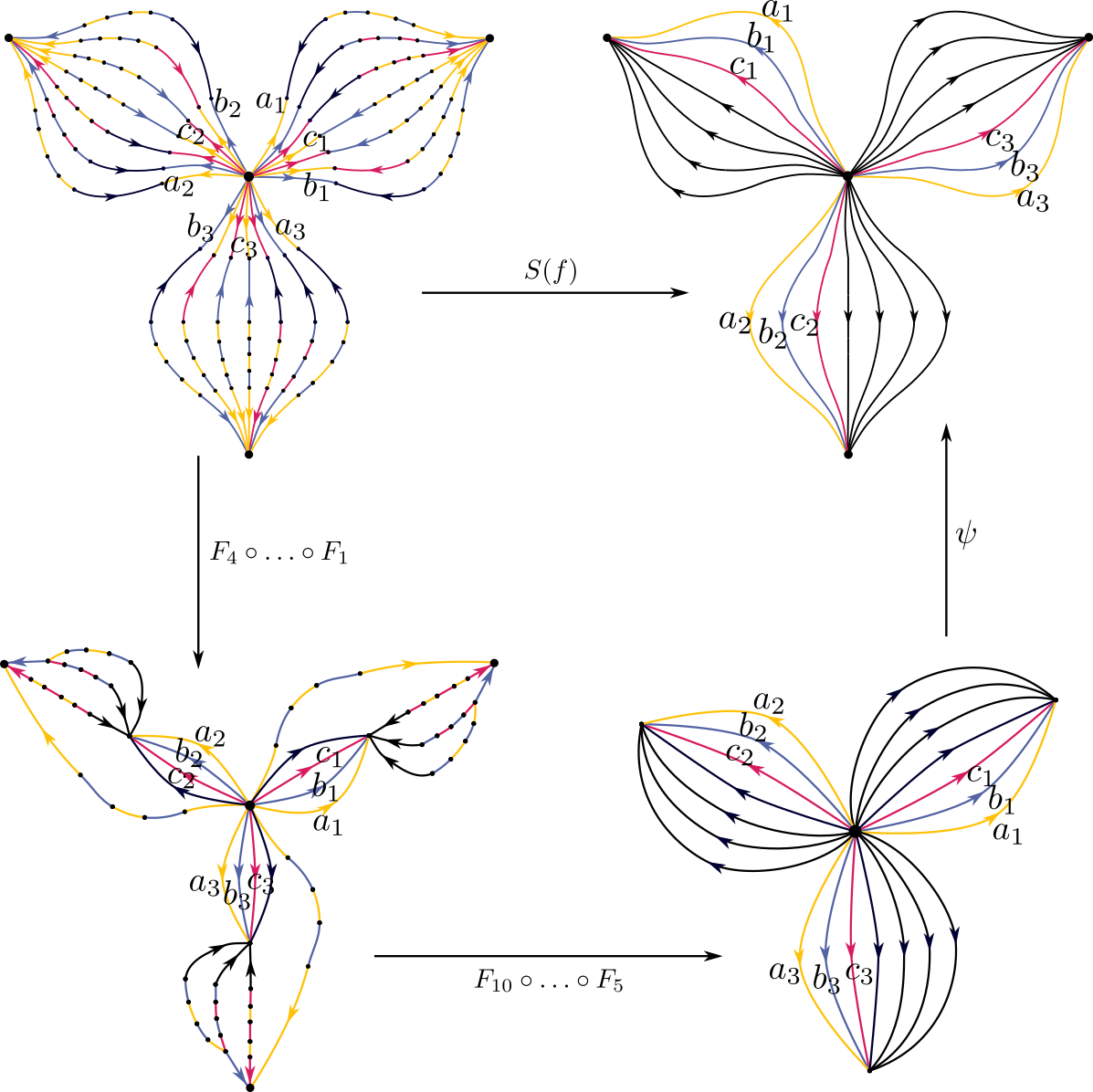}}
       \caption{The decomposition of $S(f)$. Note the immersion $\psi$ is the identity map up to permuting petals or edges within petals.}
       \label{fig:splitmapfinalstep}
   \end{figure}
   
   Finally, we perform the analogous folds $F_5, \ldots, F_{10}$ on each of the other $n-1$ petals in the rightmost graph of \Cref{fig:splitmapsecondstep}. The resulting map $\psi$ of \Cref{thm:stallings} is the identity map on $S(\ast_n)$ up to permuting the petals and permuting edges within a petal (a graph automorphism), which is indeed a homotopy equivalence. Since only folds of Type I were performed in $F_1, \ldots, F_{10}$, we conclude that $S(f)$ is a homotopy equivalence. See \Cref{fig:splitmapfinalstep} for the summary.
    \end{proof}

\section{Conclusion: Proof of Thurston's Theorem}\label{sec:Conclusion}

It remains to show that for any weak Perron or Perron number $\lambda$ that $S(f_\lambda)$ is ergodic. Note that the proof of this fact was missing in Thurston's paper. We first prove the case where $\lambda$ is Perron and then extend to weak Perron numbers.

We retain the notation from \Cref{ssec:AST_Uniform_Expanding} so that $\ast_n =\ast_{mM}$ consists of $m$ families of $M$ tips.
We call each family of $M$ tips a \emph{fan}. We label the tips by $I=(s_k,i)$ where $k \in \{1,\ldots,m\}$ and $i \in \{1,\ldots,M\}$. Also, when $i<M$ for $I=(s_{k},i)$ we write $I+1 = (s_{k},i+1)$. 

\begin{THM} \label{LEM:PerronMixing}
Let $\lambda$ be a Perron number and $f_\lambda$ be the star map constructed in \Cref{thm:astmapPerron}. Then its split map $S(f_\lambda):S(\ast_n) \to S(\ast_n)$ is mixing, i.e., for an edge $y_I$ in $S(\ast_n)$, there exists an $L$ such that $S(f_\lambda)^L(y_I) \supseteq S(\ast_n)$.
\end{THM}

There are two key aspects of our constructions thus far that will be pivotal in the proof of this lemma: the way the star map $f_\lambda$ constructed in the proof of \Cref{thm:astmapPerron} acts on the edges $(s_k,M)$ of $\ast_n$ (the last tip in each fan), and the way protoype maps act on edges in the prototype graph, with particular focus on the fact that for large $m>0$, $\phi_{3+2m}(a) = aG(aB)^ma$ maps over $aB$ a large number of times. Using these two facts together, we break down the proof of \Cref{LEM:PerronMixing} into the following five steps: 

\begin{itemize}
\item \textbf{Step 1}: We first show that the image of $y_I$ under a large enough power of $S(f_\lambda)$ contains \emph{some} $a$-edge in $S(\ast_n)$.
\item \textbf{Step 2}: Then we show that a bigger power of $S(f_\lambda)$ applied to $y_I$ contains \emph{all} $a$-edges in $S(\ast_n)$.
\item \textbf{Step 3}: Next, we show that any further power of $y_I$ under $S(f_\lambda)$ will \emph{still} contain every $a$-edge in $S(\ast_n)$. 
\item \textbf{Step 4}: We use this to conclude that there is a larger power of $S(f_\lambda)$ so that the image of $y_I$ contains all $a$- {and} $g$-edges in $S(\ast_n)$. 
\item \textbf{Step 5}: We then obtain all of the $b$-, $e$-, $c$-, $d$-, and $f$-edges in $S(\ast_n)$ using further powers of $S(f_\lambda)$, in that order, to obtain the mixing conclusion.
\end{itemize}

\begin{proof}
In what follows, we use $\Gamma$ to denote $\ast_{n}=\ast_{mM}$. We drop the subscript $\lambda$ from $f_\lambda$ and simply call its split $S(f)$.
Take a random edge in $S(\Gamma)$ labelled by $y_I$, where $y \in \{a, b, \ldots, g\}$
and $I =(s_k,i)$ for some $k \in \{1,\ldots,m\}$ and $i \in \{1,\ldots,M\}$. Then, there exists $p\geq 0$ such that $S(f)^p(y_I) = y_K$ where $K = (s_k,M)$. To be precise, we can take $p=M-i$. That is to say, we shift $y_I$ along its fan until it lives on the prototype corresponding to the last edge in this fan. Here we are using the fact that when $i < M$ and $I=(s_k,i)$, then $S(f)$ sends $y_I$ to $y_{I+1}$. We do this so that the next iterate of $S(f)$ applied to $y_I$ involves the application of a non-identity prototype map.

 Recall the prototype maps $\phi_{3+2m}$ for $m\geq 0$ are given by:
\[
    \phi_{3 + 2m} :
    \begin{cases}
    a &\mapsto aG(aB)^ma, \\
    b &\mapsto bD(bC)^mb, \\
    c &\mapsto cF(cA)^mc, \\
    d &\mapsto aB(aB)^ma, \\
    e &\mapsto cB(aB)^ma, \\
    f &\mapsto aC(aB)^ma, \\
    g &\mapsto bE(bA)^mb.
    \end{cases}
\]
\textbf{Step 1:} Since $f((s_k,M))$ always maps over at least 5 edges, we use the prototype maps $\phi_{3+2m}$ with $m \ge 1$ to construct the split map $S(f)$ from $f$. In addition, the image of an edge under such prototypes will always contain a portion of the form $(wZ)^m$ where $w,z \in \{a,b,c\}$ are distinct. Recall that $f((s_k,M))$ in $\Gamma$ maps over $(s_k,1)$ an even number of times, then maps over $(s_\ell,1)$ for all $\ell \neq k$ an even number of times in some order, then maps over $(s_k,2)$ exactly twice, and finishes by mapping over $(s_{k},n_0 + 1)$ exactly once. Therefore, the image of an edge in $S(\Gamma)$ under $S(f)$ will map over the $w$- and $z$-edges of the prototypes corresponding to $(s_\ell,1)$ for all $\ell \neq k$ and $(s_k,2)$.

From this, we conclude that, so long as $y \neq b$, $S(f)^{p+1}(y_I)$ contains $a_J$ for some $J$; in fact, this image contains several $a$-edges. This is due to the fact that the image of every other edge besides $b$ in the prototype contains word of the form $(wZ)^m$ where either $w$ or $z$ is $a$. If $y = b$, then $\phi_{3+2m}(b) = bD(bC)^{m}b$, so that $S(f)^{p+1}(y_I)$ contains $d_{(s_k,1)}$. Therefore, $S(f)^{p+2}(y_I)$ contains $d_{(s_k,2)}$, $S(f)^{p+M}(y_I)$ contains $d_{(s_k,M)}$, and $S(f)^{p+M+1}(y_I)$ contains $a_J$ for $J=(s_k,1)$ since $\phi_{3 +2m}(d)= aB(aB)^ma$.

\textbf{Step 2a:} However, we claim that since there is a power of $S(f)$ so that the image of $y_I$ contains $a_{J_0}$ for \emph{some} $J_0$, then there is a further power, call it $S(f)^q$, so that the image of $y_I$ contains $a_J$ for {all} $J = (s_j,1)$, where $j=1,\ldots,m$. Note that this is still a subset of the set of \emph{all} $a$-edges in $S(\ast_n)$. Without loss of generality, assume $J_0= (s_\ell,M)$ for some $\ell$ (making the second component $M$ can be achieved by applying a few more iterates of $S(f)$). Then, the claim follows from the fact that $\phi_{3+2m}(a) = {a}G(aB)^ma$ so that $S(f)(a_{J_0})$ first maps over $a_JG_J$ for $J = (s_\ell,1)$ and then maps over $a_JB_J$ for $J = (s_j,1)$ for all $j \neq \ell$. In fact, after mapping over $a_JB_J$ for $J = (s_j,1)$ and $j \neq \ell$, $S(f)(a_{J_0})$ then maps over $a_JB_J$ for $J=({s_{\ell}},2)$. The fact that the image contains $a_{({s_{\ell}},2)}$ is essential for showing that the map is mixing and not just ergodic as we will now see.

\begin{figure}[ht!]
    \centering
    \makebox[\textwidth][c]{
        \includegraphics[width=\textwidth]{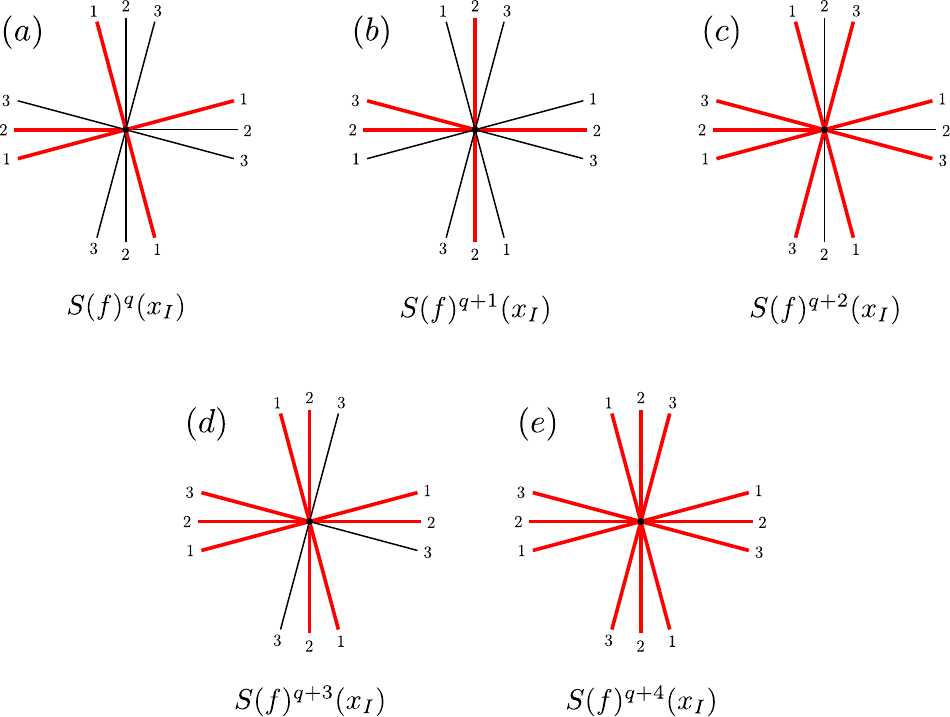}
    }
    \caption{$M=3,m=4$. The $K$-th edge in $\G$ is thickened and colored in red when $S(f)^n(y_I)$ contains $a_K$. The figure illustrates how the subsequent powers of $S(f)^q$ will map $y_I$ over \emph{all} edges of the form $a_{J}$.
    }
    \label{fig:SplitMapMixing1}
\end{figure}

 \textbf{Step 2b:} Given the fact that $S(f)^q(y_I)$ contains $a_J$ for all $J = (s_j , 1)$ and $J=(s_\ell, 2)$ for some $\ell$, we now show that another further power of $S(f)$ applied to $y_I$ contains \emph{all} $a$-edges of $S(\ast_n)$.  We outline the argument in the bullet points below for ease of readability. 
\begin{itemize}
\item First, $S(f)^{q+1}(y_I)$ contains $a_J$ where $J = (s_j,2)$ for all $j$ and for $J = (s_\ell,3)$. 
\item Then, $S(f)^{q+M-2}(y_I)$ contains $a_J$ where $J = (s_j,M-1)$ for all $j$ and for $J = (s_\ell,M)$. 
\item In the next iterate of $S(f)$, $a_{(s_j,M-1)}$ maps to $a_{(s_j,M)}$. Additionally, $a_{(s_\ell, M)}$ maps over $a_J$ where $J = (s_j,1)$ for all $j$ and for $J = (s_\ell,2)$. In summary, $S(f)^{q+M-1}(y_I)$ contains $a_J$ where $J = (s_j,1), (s_j,M)$ for all $j$ and for $J = (s_\ell,2)$. See \Cref{fig:SplitMapMixing1}$(c)$.
\item Then, $S(f)^{q+M}(y_{I})$ contains $a_{J}$ where $J=(s_{j},1)$ and $(s_{j},2)$ for all $j$ and for $J=(s_{\ell},3)$. See \Cref{fig:SplitMapMixing1}$(d)$.
\item Applying $S(f)^M$ once more, we see that $S(f)^{q+2M}(y_I)$ contains all $a_J$ where $J=(s_j,1)$, $(s_j,2)$, $(s_j,3)$, for all $j$ and for $J=(s_\ell,4)$.
\item Continuing in this fashion, $S(f)^{q+(M-2)M}(y_I)$ contains all $a_J$ where $J=(s_j,1),\ldots,(s_j,M-1)$ for all $j$ and for $J=(s_\ell,M)$.
\item  Since $S(f)(a_{(s_\ell,M)})$ contains all $a_J$ of the form $J=(s_j,1)$, it follows that $S(f)^r(y_I)$, for $r=q+(M-2)M+1$, contains $a_J$ for all $J$.
\end{itemize}

Thus, we have that there exists $r$ such that $S(f)^r(y_I)$ contains $a_J$ for all $J$ concluding Step 2.

\textbf{Step 3:} We claim that any further power of $y_I$ under $S(f)$  will
still contain every $a$-edge in $S(\ast_n)$. In particular, $S(f)(a_J) = a_{J+1}$ for all $J=(s_j,i)$ when $i < M$, and $S(f)(a_{(s_\ell,M)})$ contains all $a_{(s_j,1)}$. In this way, we do not lose any $a$-edges in $S(f)^r(y_I)  \subset S(\ast_n)$ by applying further powers of $S(f)$. See \Cref{fig:SplitMapMixing2}.

\begin{figure}[ht!]
    \centering
    \makebox[\textwidth][c]{
        \includegraphics[width=.6\textwidth]{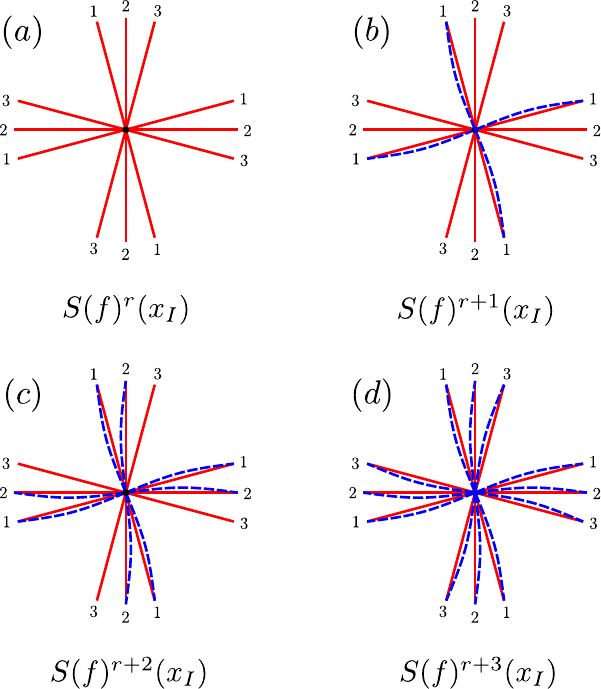}
    }
    \caption{Red(with non-dashed) and blue(with dashed) edges in $S(\G)$ are $a$- and $g$-edges in the image of $y_I$ via corresponding power of $S(f)$ written below each graph. The figure illustrates how a subsequent power of $S(f)^r$ will map $y_I$ over \emph{all} edges of the form $a_{J}$ and $g_{J}$. Note the red edges persist throughout the process.}
    \label{fig:SplitMapMixing2}
\end{figure}

\textbf{Step 4:} Additionally, we claim that for each $j$, $S(f)(a_{(s_j,M)})$ also contains $g_{(s_{j},1)}$. This follows from the fact that the first two letters in the image of $a$ under any prototype map (except $\phi_1$; the identity) are $aG$ and the fact that $f((s_j,M))$ first maps over $(s_j,1)$ at least twice. Therefore, $S(f)^{r+1}(y_{I})$ contains $g_{J}$ for all $J=(s_{j},1)$ (See \Cref{fig:SplitMapMixing2}$(b)$) and $S(f)^{r+2}(y_I)$ contains $g_J$ for all $J = (s_j,1), (s_j,2)$ since $g_{(s_j,1)}$ maps to $g_{(s_j,2)}$ and $a_{(s_j,M)}$ maps over $g_{(s_j,1)}$. See \Cref{fig:SplitMapMixing2}$(c)$. Finally, $S(f)^{r+M}(y_I)$ contains $a_J$ and $g_J$ for all $J$. 

\textbf{Step 5:} We continue this procedure to obtain, in order, every $b_J$ and $e_J$, then $c_J$ and $d_J$, and finally $f_J$ edge in $S(\ast_n)$. This order is determined by the first two letters in the images of $\{a, b, \ldots, g\}$ under the prototype maps. Therefore, there exists an $L$ such that $f^L(y_I)$ contains all of $S(\ast_n)$.  
\end{proof}

Proceeding to the case of weak Perron numbers $\lambda$, we have:

\begin{THM} \label{thm:WeakPerronErgodic}
Let $\lambda$ be a weak Perron number and $f_\lambda$ be the star map constructed in \Cref{lem:WeakPerron_Expanding_Ast}. Then its split map $S(f_\lambda):S(\ast_n) \to S(\ast_n)$ is ergodic, i.e., for each pair of edges $y_I,z_J$ in $S(\ast_{k})$, there exists an $L$ such that $S(f_\lambda)^L(y_I) \supset z_J$.
\end{THM}

\begin{proof}
By \Cref{prop:PnWP}, there is an integer $N$ such that $\lambda^N$ is Perron. We use $\mu$ to denote $\lambda^N$ in the remainder of this proof for notational simplicity. Let $f_{\mu}:\ast_{n} \to \ast_{n}$ be the uniformly $\mu$-expanding star map from \Cref{thm:astmapPerron} and $f_\lambda:\ast_{Nn} \to \ast_{Nn}$ be the uniformly $\lambda$-expanding star map constructed using $f_\mu$ as in the proof of \Cref{lem:WeakPerron_Expanding_Ast}.

Label the edges of $\ast_n$ as $j \in \{1,\ldots,n\}$, and the edges of  $\ast_{Nn}$ as $(i,j)$, where $i \in \{0,\ldots,N-1\}$ and $j \in \{1,\ldots,n\}$. Then label the edges of $S(\ast_n)$ as $y_j$ and label the edges of $S(\ast_{Nn})$ as $y^i_j$'s, where $y \in \{a ,b, \ldots, g\}$, $i \in \{0,\ldots,N-1\}$ and $j \in \{1,\ldots,n\}$.

For each $i=0,\ldots,N-1$, define $C^i$ as the subgraph of $S(\ast_{Nn})$ consisting of the edges of the form $y^i_j$, which is homotopy equivalent to $S(\ast_n)$. Call $C^0,\ldots,C^{N-1}$ the \emph{fans} of $S(\ast_{Nn})$. Additionally, for a path $\gamma$ in $S(\ast_n)$, we use $[\gamma]^i$ to denote the copy of this path in the $i$-th fan $C^i$ of $S(\ast_{Nn})$.

Recall that $S(f_\mu)$ is mixing (by \Cref{LEM:PerronMixing}) and that $f_\lambda$ is constructed to simply shift $C^i$ to $C^{i+1}$ for all $i < N-1$ and constructed to apply $f_\mu$ to $C^{N-1}$ followed by shifting it to $C^{0}$.  We claim that $S(f_\lambda)$ is ergodic.

Now pick two edges $y^i_{j}$ and $z^{i'}_{j'}$ in $S(\ast_{Nn})$. Consider the two corresponding edges $y_j$ and $z_{j'}$ of $S(\ast_n)$. By the fact that $S(f_{\mu})$ is mixing (and therefore ergodic), there exists $p$ such that $S(f_{\mu})^p(y_j) \supset z_{j'}$. We will show that
\[
  S(f_{\lambda})^{Np-i+i'}(y_j^i)\supset z_{j'}^{i'}.
\]
The following diagram breaks down how the power $Np-i+i'$ of $S(f_\lambda)$ is obtained, where the numbers over the arrows denote the power of $S(f_\lambda)$ being applied. We will carefully explain each arrow in what follows:
\[
    y_j^i \quad \overset{N-i}{\longmapsto} \quad
    [S(f_\mu)(y_j)]^0 \quad \overset{N(p-1)}{\longmapsto} \quad
    [S(f_\mu)^p(y_j)]^0 \quad \overset{i'}{\longmapsto} \quad
    [S(f_\mu)^p(y_j)]^{i'} \supset z_{j'}^{i'}
\]

By the construction of $f_\lambda$ from $f_{\mu}$, $S(f_\lambda)$ simply translates $C^i$ to $C^{i+1}$ when $i < N-1$. Thus, we first move $y_j^i$ to $y_j^{N-1}$ using $S(f_\lambda)^{N-1-i}$. Applying $S(f_\lambda)$ one more time brings $y_j^{N-1}$ to $[S(f_\mu)(y_j)]^0$. This is due to the fact that $S(f_\lambda)$ amounts to applying $S(f_\mu)$ to $C^{N-1}$ and then translating the image to $C^0$. In summary, $S(f_\lambda)^{N-1 -i + 1}(y_j^i) = S(f_\lambda)^{N-i}(y_j^i) = [S(f_\mu)(y_j)]^0 \subset C^0$.

Next, notice that applying $S(f_\lambda)^N$ to this path amounts to applying $S(f_\mu)$ one time to $C^0$. Thus, $$S(f_\lambda)^{N-i+N(p-1)}(y_j^i) = S(f_\lambda)^{Np-i}(y_j^i) = [S(f_\mu)^p(y_j)]^0.$$ Moreover, $[S(f_\mu)^p(y_j)]^0 \supset z_{j'}^0$ given the fact that $S(f_\mu)^p(y_j) \supset z_{j'}$. Lastly, we apply $S(f_\lambda)^{i'}$ to bring $z_{j'}^0$ to $ z_{j'}^{i'}$. Thus, $S(f_\lambda)^{Np - i + i'} (y_j^i) \supset z_{j'}^{i'}$, proving that $S(f_\lambda)$ is ergodic.
\end{proof}

\begin{RMK} The map $S(f_\lambda)$ from
the construction in the proof of 
\Cref{thm:WeakPerronErgodic} is \emph{not} mixing. Indeed, $S(f_{\lambda})^k(y^{i}_{j})$ is
completely contained in $C^{m}$ where 
$m \equiv i+k$ modulo $N$.
\end{RMK}

We conclude the paper with the proof of Thurston's main theorem  \Cref{thm:thurston_main}.

\begin{proof}[Proof of \Cref{thm:thurston_main}]
    The forward direction is covered in the last paragraph of \Cref{ssec:BG_PF_Theorem}. That is, if $h$ is the topological entropy of an ergodic traintrack representative of an outer automorphism of a free group, then $e^h$ is a weak Perron number.
    
    Conversely, let $\lambda$ be a weak Perron number. Then using \Cref{lem:WeakPerron_Expanding_Ast} we find a uniformly $\lambda$-expanding star map $f_\lambda:\ast_k \to \ast_k$ for some $k$. Splitting $f_\lambda$ we obtain our desired map $S(f_\lambda):S(\ast_k)\to S(\ast_k)$. This is a traintrack map by \Cref{prop:splitttmap} with respect to the traintrack structure given in \Cref{ssec:SPL_Graphs}, and is indeed a topological representative of an outer automorphism by \Cref{prop:splithe}. Also, $S(f_\lambda)$ is still uniformly $\lambda$-expanding by \Cref{prop:splitexpanding}, and is ergodic by \Cref{thm:WeakPerronErgodic}. Finally, by \Cref{prop:entropyOfUniformMaps} the topological entropy of $S(f_\lambda)$ is indeed $\log \lambda$, concluding the proof.
\end{proof}
\bibliography{bib}
\bibliographystyle{plain}
\end{document}